\documentclass[12pt
]{article}

\usepackage{amssymb}
\usepackage{amsmath}
\usepackage{mathrsfs}
\usepackage{theorem}

\usepackage{newtxtext}

\newtheorem{theorem}{Theorem}
\newtheorem{lemma}{Lemma}
\newtheorem{corollary}{Corollary}
\newenvironment{proof} {{\bf Proof.}}{\hfill \fbox{}\\ \smallskip}

{\theorembodyfont{\rmfamily} \newtheorem{remark}{Remark}}
{\theorembodyfont{\rmfamily} }
{\theorembodyfont{\rmfamily} }

\newcommand{\C}{\mathbb{C}}

\newcommand{\n}{\nu}
\newcommand{\h}{\eta}
\newcommand{\D}{\Delta}
\newcommand{\del}{\delta}
\newcommand{\de}{\partial}
\newcommand{\ga}{\gamma}
\newcommand{\ep}{\varepsilon}

\newcommand{\vf}{\varphi}

\newcommand{\p}{\psi}

\newcommand{\si}{\sigma}

\newcommand{\al}{\alpha}

\newcommand{\la}{\lambda}
\newcommand{\LA}{\Lambda}
\newcommand{\ze}{\zeta}

\newcommand{\Th}{\Theta}
\newcommand{\Om}{\Omega}
\newcommand{\om}{\omega}
\newcommand{\ro}{\varrho}
\newcommand{\dive}{\mathop{\rm div}\nolimits}
\newcommand{\Dive}{\mathop{\rm Div}\nolimits}

\newcommand{\vertiii}[1]{{\left\vert\kern-0.25ex\left\vert\kern-0.25ex\left\vert #1 
    \right\vert\kern-0.25ex\right\vert\kern-0.25ex\right\vert}}

\newcommand{\lan}{\langle}
\newcommand{\ran}{\rangle}

\newcommand{\essinf}{\mathop{\rm ess\, inf}}
\newcommand{\esssup}{\mathop{\rm ess\, sup}}

\renewcommand\leq{\leqslant}
\renewcommand\geq{\geqslant}

\newcommand\dist{\text{dist}}
\newcommand\spt{\text{spt}\,}

\newcommand{\R}{\mathbb R}

\newcommand\Cspt{\mathaccent"017{C}}
\newcommand{\Hspt}{\mathaccent"017{H}}

\renewcommand\Re{\mathop{\mathbb R \rm{e}}\nolimits}

\newcommand\elle{\mathop{\mathscr L}\nolimits}
\newcommand\A{\mathop{\mathscr A}\nolimits}

\title{Criterion for the functional dissipativity of the Lam\'e operator}

\author{A. Cialdea
\thanks{Department of Mathematics, Computer Sciences and Economics,
University of Ba\-si\-li\-ca\-ta, V.le dell'Ateneo Lucano, 10, 85100 Potenza, Italy.
 \textit{email:}
cialdea@email.it.}\and
V. Maz'ya 
\thanks{Department of Mathematics, Link\"oping University,
SE-581 83, Link\"oping, Sweden.
Peoples' Friendship University
 of Russia (RUDN University);
6 Miklukho-Maklaya St, Moscow, 117198, Russian Federation.
\textit{email}: vladimir.mazya@liu.se.}
}
\date{}    % optional

\begin{document}

\maketitle

\begin{flushright}
     \textit{Dedicated to Natasha and Sasha Movchan on the occasion of their jubilee}
 \end{flushright}
 
 \bigskip
 \bigskip

\textbf{Abstract.} After introducing the concept 
of  functional dissipativity of the Dirichlet problem in a domain $\Omega\subset \R^N$ for  systems of partial differential operators of the form $\de_{h} (\A^{hk}(x)\de_{k})$ ($\A^{hk}(x)$ being $m \times m$ matrices with complex 
valued $L^\infty$ entries), we find necessary and sufficient conditions for the functional dissipativity   
 of the two-dimensional Lam\'e system.  As an application of our theory we provide  two regularity results for the displacement vector in the $N$-dimensional equilibrium problem, when the body  is fixed along its boundary.

 \section{Introduction}
 
The fundamentals of mathematical elasticity, as discussed in the classical monographs \cite{lurie,timoshenko}, have a wide range of important applications in modelling of multi-scale solids, including inhomogeneous structures, which may consist of different elastic constituents. Formally, such problems are described by a system of partial differential equations with variable coefficients. In many applications, it is preferred to use formulations where interface transmission conditions are formulated on boundaries separating different constituents of an inhomogeneous solid. On the other hand, there are important examples of micro-structured inhomogeneous solids, where the number of constituents is large, and hence  homogenisation approximations are commonly used. If there are no periodicity assumptions on the micro-structure, and the micro-structured solids includes many inclusions of different scales, an effective and rigorous method of meso-scale asymptotic approximations has been introduced in \cite{MMN}, which provides an excellent alternative to the direct homogenisation. The homogenisation procedures on their own have also proved to be very popular and, with the variable micro-structure, may lead to the equations with variable coefficients, which describe the dependence of the elastic coefficients on the spatial variables in the constitutive equations. The mathematical analysis of such problems and their numerical treatment bring additional challenges, which required new work. 

The present paper addresses a class of problems of mathematical elasticity, including the multi-dimensional formulations,  where the governing equations incorporate the variable coefficients, and the analysis is presented in the context of functional dissipativity of linear operators.
 The concept of functional dissipativity of a linear operator was recently introduced in \cite{CM2021}.
 If $A$ is the scalar second order partial differential operator
$\nabla(\A \nabla )$,
where $\A$ is a square matrix whose entries are complex valued $L^{\infty}$-functions defined in the domain $\Om\subset\R^{N}$,  we say that $A$
is functional dissipative with respect to a given positive function $\vf:\R^{+}\to \R^{+}$ if 
$$
\Re \int_\Om \lan \A \nabla u, \nabla(\vf(|u|)\, u)\ran\, dx \geq 0
$$
for any $u\in \Hspt^{1}(\Om)$ such that $\vf(|u|)\, u \in \Hspt^{1}(\Om)$. 

As explained in the introduction of \cite{CM2021}, a motivation for the study of this concept comes from the decrease of the  Luxemburg 
norm of solutions of the Cauchy–Dirichlet problem
$$
 \begin{cases}
u'=Au \\ 
u(0)=u_{0} \, . 
\end{cases}
$$

Here the  Luxemburg
norm is taken in the Orlicz space
  of functions $u$ for which
there exists $\al>0$ such that
$$
\int_{\Om}\Phi(\al\, |u|)\, dx < +\infty\, ,
$$
where the Young function $\Phi$ is related to $\vf$ by
$$
\Phi(s)=\int_{0}^{s} \si\, \vf(\si)\, d\si\, .
$$
 
The functional dissipativity is an extension of the concept of $L^p$-dis\-si\-pa\-tiv\-i\-ty, which is obtained
taking $\vf(t)=t^{p-2}$ ($1<p<\infty$).  In a series of papers \cite{CM2005,CM2006,CM2013,CM2018} we have studied the problem of characterizing the
$L^p$-dissipativity of scalar and matrix partial differential operators. In the monograph \cite{CMbook} this theory is considered in the more general frame of semi-bounded operators. For a short survey of our results we refer to 
the introduction of \cite{CM2021}.

 %In \cite{CM2021}  a necessary and sufficient conditions for the functional dissipativity of operator $A$ 
 %was given.
 
 The aim of the present paper is 
 to study the functional dissipativity of the  
 two-dimensional Lam\'e operator
  \begin{equation}\label{opelast}
Eu= \nabla\cdot(\la(x)\, \dive u\, I  + \mu(x)\, (\nabla u + (\nabla u)^{T})\, .
\end{equation}
 
 The Lam\'e parameters $\la$ and $\mu$ are supposed to be real valued $L^{\infty}$ functions satisfying the usual 
ellipticity conditions (see \eqref{eq:lameinf} below). 
Previously  we have considered the case of constant Lam\'e parameters
and proved that 
 \begin{equation}   \label{opelastconst}
    Eu=\mu\D u + (\la+\mu)\nabla \dive u\, ,
\end{equation}
 is $L^{p}$-dissipative if and only if
$$
\left(\frac{1}{2}-\frac{1}{p}\right)^{2}\leq 
\frac{2(\n-1)(2\n-1)}{ (3-4\n)^{2}}\, ,
$$
where $\n$ is the Poisson ratio (see \cite[Th. 3, p.244]{CM2006}).  
Note that this condition can be written in terms of Lam\'e constants as
$$
\left(1-\frac{2}{p}\right)^{2}\leq  1- \left(\frac{\la + \mu}{\la+3\mu}\right)^2 .
$$

 In the first part of the present paper we study
 the functional dissipativity of
 a general system of partial differential operators of the form
\begin{equation}\label{eq:op1}
A=\de_{h} (\A^{hk}(x)\de_{k})
\end{equation}
where $\A^{hk}(x)=\{a^{hk}_{ij}(x)\}$ are 
$m\times m$ matrices 
whose elements are complex valued $L^{\infty}$-functions  functions defined 
in a domain $\Om\subset\R^{N}$ 
$(1\leq i,j\leq 
m,\ 1\leq h,k \leq N)$.  The Lam\'e system  is obtained taking
$$
a^{hk}_{ij}(x)= \la(x) \del_{ih}  \del_{jk} + \mu(x)  ( \del_{ij}  \del_{hk}   + \del_{ik}  \del_{hj})\, .
 $$
%$(1\leq i,j, h,k \leq 2)$.

Concerning the general systems \eqref{eq:op1}, the operator $A$ is functional dissipative (or $L^\Phi$-dissipative) if
$$
\Re \int_{\Om} \lan \A^{hk} \de_{k} u, \de_{h}(\vf(|u|)\, u)\ran\, dx \geq 0
$$
for any $u\in [\Hspt^{1}(\Om)]^m$  such that $\vf(|u|)\, u\in 
[\Hspt^{1}(\Om)]^m$. We say also that the operator $A$ is strict functional dissipative if there exists $\kappa>0$ such that
$$
\Re \int_{\Om} \lan \A^{hk} \de_{k} u, \de_{h}(\vf(|u|)\, u)\ran\, dx \geq 
\kappa \int_{\Om} | \nabla(\sqrt{\vf(|u|)}\, u)|^2 dx
$$
 for any $u\in [\Hspt^{1}(\Om)]^m$  with $\vf(|u|)\, u\in 
[\Hspt^{1}(\Om)]^m$. 

The last concept is strictly related to the concept of $p$-elliptic operator. This 
was considered in a series of papers by  Carbonaro and Dragi\v{c}evi\'{c} \cite{CD20201,CD20202},  Dindo\v{s} and Pipher \cite{DP20191,DP20192,DP20201,DP20202},
Egert \cite{egert}.
It is worthwhile to remark that, if the partial differential operator has no lower order terms, the concepts of $p$-ellipticity
and strict  $L^p$-dissipativity coincide. 
 Our results show that the operator $A$ is strict  $L^p$-dissipative, i.e. $p$-elliptic,  if and only if there exists $\kappa>0$ such that
 $A-\kappa\Delta$ is $L^p$-dissipative (see Corollary \ref{co:co1} below).

Concerning  Lam\'e system with constant Lam\'e parameters,
in \cite[Corollary 1, p.246]{CM2006}) 
we proved also  that there exists $\kappa>0$ such that $E-\kappa\Delta$ is $L^p$-dissipative if and only if
$$
\left(\frac{1}{2}-\frac{1}{p}\right)^{2}<
\frac{2(\n-1)(2\n-1)}{ (3-4\n)^{2}}\, ,
$$
i.e.
$$
\left(1-\frac{2}{p}\right)^{2}<  1- \left(\frac{\la + \mu}{\la+3\mu}\right)^2 .
$$

As remarked before, this is equivalent to say that $E$ is strict  $L^p$-dis\-si\-pa\-tive, i.e. $E$ is $p$-elliptic.  The last result was
recently extended to variable Lam\'e parameters by Dindo\v{s}, Li and Pipher \cite{DLP}.
It must be pointed out that  these Authors introduce an auxiliary function  $r(x)$  (see \cite[formula (85)]{DLP}) which
generates some first order terms in the partial differential operator. In the definition of $p$-ellipticity these terms do not play any role,
while they have some role in the  dissipativity. Therefore our and their results do not seem to be completely
equivalent.

The main result of the present paper is that, assuming that the BMO seminorm
of the function  $\mu^2\, (\la + 3\mu)^{-1}$ is sufficiently small, elasticity operator \eqref{opelast} is strict functional dissipative if and only if 
$$
\LA^{2}_{\infty} <    1 - \esssup_{x\in\Om}\left(\frac{\la + \mu}{\la+3\mu}\right)^2,
$$
where $\LA^{2}_{\infty}=\sup_{t>0} \LA^{2}(t)$
and $\LA$ is the function defined by the relation
$$
\LA\left(s\sqrt{\vf(s)}\right)= - \frac{s\, \vf'(s)}{s\,\vf'(s)+2\, 
\vf(s)}\,  .
$$
For the theory of BMO functions we refer to Stein \cite[Chapter IV]{stein}.

This paper is organized as follows. In Section \ref{sec:prelim} we specify the class of functions $\vf$ we are going to consider, 
 introduce some related functions and recall some results obtained in \cite{CM2021}.
 
 Section \ref{sec:CNES} is devoted to prove necessary and sufficient conditions for the functional dissipativity
 of the general system of the second order in divergence form \eqref{eq:op1}. Specifically we prove the equivalence between the  functional dissipativity 
 (strict functional dissipativity)  of such an operator 
 and the positiveness (strict  positiveness) of the real part of a certain  form in $[\Hspt^{1}(\Om)]^{m}$.
 
 In Section \ref{sec:CN} we give  algebraic necessary conditions for the functional dissipativity and the strict functional dissipativity
  of  a general system when $N=2$.  
 We remark that we prove these results under the additional assumption that the function $|s\, \vf'(s)/ \vf(s)|$ is not decreasing.
 
 The main result concerning  the  strict functional dissipativity of two-di\-men\-sion\-al elasticity operator is proved
 in Section \ref{sec:elast}.

As an application of our theory, in the last Section \ref{sec:applic} we provide two regularity results
for the energy solution of the Dirichlet problem for Lam\'e system with zero data on the boundary. This represents the equilibrium problem
in linear elasticity for a body which is fixed along its boundary. We mention that 
Dindo\v{s}, Li and Pipher \cite{DLP} obtained some regularity results for solutions of the
Lam\'e system which are
of a different nature.

 \section{Preliminaries}\label{sec:prelim}
 
 In this Section we recall some definitions and results obtained in \cite{CM2021}.
 
 Let $\Om$ be an open set in $\R^{N}$.
As usual, by $\Cspt^\infty(\Om)$ we denote the space of  complex valued $C^{\infty}$ functions
having compact support in $\Om$
and by
 $\Hspt^1(\Om)$ the closure of $\Cspt^\infty(\Om)$ in the norm
 $$
 \int_{\Om}(|u|^2  + |\nabla u|^2)dx, 
 $$
$\nabla u$ being the gradient of the function $u$. For the basic facts on functional spaces
used in the following one can consult, e.g.,  Brezis \cite{brezis}.

The inner product either in
$\C^{N}$ or in $\C$ is denoted by $\lan \cdot, \cdot \ran$  and the bar denotes complex conjugation.
 
% \subsection{The function $\vf$}
From now on we assume that $\vf$ is a positive function satisfying 
 the following conditions
\renewcommand{\labelenumi}{(\roman{enumi})}
\renewcommand{\theenumi}{(\roman{enumi})}
\begin{enumerate}
	\item\label{item1} $\vf \in C^{1}((0,+\infty))$;
	\item\label{item2} $(s\, \vf(s))'>0$ for any $s>0$;
	\item the range of the strictly increasing function $s\, \vf(s)$ is  $(0,+\infty)$;	
	\item\label{item4} there exist two positive constants $C_{1}, C_{2}$  and a real number $r>-1$ such that
$$
C_{1} s^{r}\leq (s\vf(s))' \leq C_{2}\, s^{r}, \qquad s\in (0,s_{0})
$$
for a certain $s_{0}>0$. If $r=0$ we require more restrictive 
conditions: there exists the finite limit $\lim_{s\to 
0^+}\vf(s)=\vf_{+}(0)>0$ 
and  $\lim_{s\to 0^+}s\, \vf'(s)=0$.
\item\label{item5} 
There exists $s_{1}>s_{0}$ such that 
$$
   \vf'(s)\geq 0  \text{ or }  \vf'(s)\leq 0 \qquad \forall\ s\geq s_{1}
   . 
   $$
\end{enumerate}

The condition \ref{item4} prescribes the behaviour of the function $\vf$ in a neighborhood of the origin,
while \ref{item5} concerns the behaviour for large $s$.

Let us denote by $t\, \psi(t)$ the inverse function of $s\, \vf(s)$. 
The functions
$$
\Phi(s)=\int_{0}^{s} \si\, \vf(\si)\, d\si, \qquad
\Psi(s)= \int_{0}^{s} \si\, \psi(\si)\, d\si
$$
are conjugate Young functions. 

\begin{lemma}[\hbox{\cite[Lemma 1]{CM2021}}]\label{lemma:new}
The function $\vf$ satisfies conditions \ref{item1}-\ref{item5}
if and only if the function $\psi$ satisfies
the same conditions with $-r/(r+1)$ instead of $r$.
\end{lemma}

We have also
\begin{equation}\label{eq:phipsi}
\sqrt{\psi(|w|)}\, w = \sqrt{\vf(|u|)}\, u\, ,
\end{equation}
where $w=\vf(|u|)\, u$
(see \cite[formula (43)]{CM2021}).

We need to introduce also some other functions.

Let $\ze(t)$  be the inverse of the strictly 
increasing function $s\sqrt{\vf(s)}$., i.e. 
$\ze(t) = \left(s\sqrt{\vf(s)}\right)^{-1}$.
The range of $s\sqrt{\vf(s)}$
is $(0,+\infty)$ and $\ze(t)$ belongs to $C^{1}((0,+\infty))$.
Define
 \begin{equation}\label{eq:defHGA}
\quad \Theta(t) = \ze(t)/t; \quad \LA(t)=t\, \Theta'(t)/\Theta(t)\, .
\end{equation}

One can prove (see \cite[formula (6)]{CM2021}) that 
\begin{equation}\label{eq:G=}
\LA\left(s\sqrt{\vf(s)}\right)= - \frac{s\, \vf'(s)}{s\,\vf'(s)+2\, 
\vf(s)}\,  .
\end{equation}

\begin{lemma}[\hbox{\cite[Lemma 2]{CM2021}}]
Let $\widetilde{\ze}(t)$ the inverse function of $t\, \sqrt{\psi(t)}$ 
and define, as in \eqref{eq:defHGA},
$$
\quad \widetilde{\Theta} (t) = \widetilde{\ze}(t)/t\,; \quad \widetilde{\Lambda}(t)=t\, 
\widetilde{\Theta}'(t)/\widetilde{\Theta}(t)\, .
$$
We have
\begin{equation}\label{eq:GAtilde}
\widetilde{\Theta} (t) = \frac{1}{\Theta(t)}\, , \qquad 
\widetilde{\Lambda}(t)= - \Lambda(t)
\end{equation}
for any $t>0$.
\end{lemma}

We write also two equalities given in \cite{CM2021}:
\begin{equation}\label{eq:H2psi}
\Theta^{2}(t)\, \vf[\ze(t)] =1\, , \quad \forall\ t>0, 
\end{equation}
and 
\begin{equation}\label{eq:tH'}
\Theta(t)\, \vf'[\ze(t)]\, [t\, \Theta'(t) + \Theta(t)] + \Theta'(t)\, \vf[\ze(t)] =
-\Theta'(t)/\Theta^{2}(t)\, , \quad \forall\ t>0.
\end{equation}

 Finally we note the following Lemma, proved in the scalar case in \cite[Lemma 3]{CM2021}. The extension
 to vector valued functions is immediate. 
 \begin{lemma}\label{lemma:fipsi}
If $u\in [H^{1}(\Om)]^m$ $([\Hspt^{1}(\Om)]^m)$ is such that $\vf(|u|)\, u\in 
[H^{1}(\Om)]^m$ $([\Hspt^{1}(\Om)]^m)$, then $\sqrt{\vf(|u|)}\, u$ belongs to $[H^{1}(\Om)]^m$ 
$([\Hspt^{1}(\Om)]^m)$.
\end{lemma}

 \section{Necessary and sufficient conditions for the functional dissipativity of general systems}\label{sec:CNES}

 Let $\Om$ be a domain of $\R^{N}$ and
let $A$ be the operator
\begin{equation}
    A=\de_{h} (\A^{hk}(x)\de_{k})
    \label{eq:A}
\end{equation}
where  $\de_{k} =\de / \de x_{k}$ and $\A^{hk}(x)=\{a^{hk}_{ij}(x)\}$ are 
$m\times m$ matrices 
whose elements are complex valued $L^{\infty}$-functions  functions defined 
in $\Om$ 
$(1\leq i,j\leq 
m,\ 1\leq h,k \leq N)$.  Here and in the sequel, we adopt the standard summation convention
on repeated indices.

The operator $A$ is  said to be $L^\Phi$-dissipative or functional dissipative if
\begin{equation}\label{eq:defdiss0}
\Re \int_{\Om} \lan \A^{hk} \de_{k} u, \de_{h}(\vf(|u|)\, u)\ran\, dx \geq 0
\end{equation}
for any $u\in [\Hspt^{1}(\Om)]^m$  such that $\vf(|u|)\, u\in 
[\Hspt^{1}(\Om)]^m$.

We say that the operator $A$ is strict  $L^\Phi$-dissipative if there exists $\kappa>0$ such that
\begin{equation}\label{eq:defdiss0str}
\Re \int_{\Om} \lan \A^{hk} \de_{k} u, \de_{h}(\vf(|u|)\, u)\ran\, dx \geq 
\kappa \int_{\Om}| \nabla(\sqrt{\vf(|u|)}\, u)|^2 dx
\end{equation}
for any $u\in [\Hspt^{1}(\Om)]^m$  with $\vf(|u|)\, u\in 
[\Hspt^{1}(\Om)]^m$. We remark that in view of Lemma \ref{lemma:fipsi} the right hand side
is finite.

We have the following Lemma

\begin{lemma}
If the operator $A$ is strict  $L^\Phi$-dissipative then 
$$
\Re \int_{\Om} \lan \A^{hk} \de_{k} u, \de_{h}(\vf(|u|)\, u)\ran\, dx \geq 
\frac{\kappa}{4} \int_{\Om}\vf(|u|) \, | \nabla u|^2 dx
$$
for any $u\in [\Hspt^{1}(\Om)]^m$  such that $\vf(|u|)\, u\in 
[\Hspt^{1}(\Om)]^m$, 
where $\kappa$ is the constant in \eqref{eq:defdiss0str}.
\end{lemma}
\begin{proof}
A direct computation shows that
\begin{equation}\label{eq:dircomp}
|\nabla(\sqrt{\vf(|u|)}\, u)|^{2} = 
\left( \frac{(\vf'(|u|)^{2} |u|^{2}}{4\, \vf(|u|)} + \vf'(|u|)\, |u|
\right) |\nabla |u||^{2} + \vf(|u|)\,  | \nabla u|^2
\end{equation}
on the set where $u\neq 0$. We can write
\begin{equation}\label{eq:nabla}
\begin{gathered}
|\nabla(\sqrt{\vf(|u|)}\, u)|^{2}  \\
= \left( \frac{(\vf'(|u|)^{2} |u|^{2}}{4\, \vf(|u|)}
 +
 \vf'(|u|)\, |u| + \vf(|u|)
\right) |\nabla |u||^{2}  \\
+
\vf(|u|) \, (  | \nabla u|^2- |\nabla |u||^{2} ).
\end{gathered}
\end{equation}

 On the other hand condition \ref{item2} implies 
$$
\frac{t\vf'(t) + 2\vf(t)}{\vf(t)} = \frac{t\vf'(t) + \vf(t)}{\vf(t)} + 1 \geq 1
$$
for any $t>0$ and then
$$
\left(\frac{t\vf'(t)}{2\vf(t)}+ 1\right)^{2}\geq \frac{1}{4}\, .
$$

From \eqref{eq:nabla} it follows that
$$
|\nabla(\sqrt{\vf(|u|)}\, u)|^{2} \geq \frac{\vf(|u|)}{4} \, |\nabla|u||^{2} +  \vf(|u|) \, (  | \nabla u|^2- |\nabla |u||^{2} )
$$
and this gives
\begin{equation}\label{eq:lastineq}
|\nabla(\sqrt{\vf(|u|)}\, u)|^{2} \geq \frac{\vf(|u|)}{4} \, | \nabla u|^2 ,
\end{equation}
which proves the Lemma.
\end{proof}

In the particular case $\vf(t)=t^{p-2}$,  Dindo\v{s} and Pipher \cite{DLP} (see also \cite{DP20192} for the scalar case)  proved that
$|\nabla (|u|^{(p-2)/2}u)|^{2}$ and  $|u|^{p-2}|\nabla u|^{2}$ are equivalent. It is  natural to ask if 
this is still true for a general $\vf$.

The answer is negative. While  \eqref{eq:lastineq} is always valid, the opposite inequality
\begin{equation}\label{eq:invdis}
|\nabla(\sqrt{\vf(|u|)}\, u)|^{2}  \leq C\,  \vf(|u|) \,  | \nabla u|^2 
\end{equation}
could fail. An example is given by $\vf(t)=\exp(t^{2})$.  It satisfies condition \ref{item1}-\ref{item5}, but
\eqref{eq:invdis} cannot hold. Indeed,
  in view of \eqref{eq:dircomp},  this  inequality for such a function $\vf$ can be written as
$$
(|u|^{4}+2\, |u|) |\nabla |u| |^{2} + | \nabla u|^2   \leq C\,  | \nabla u|^2 ,
$$
which  is impossible to hold for any $u\in [\Cspt^{\infty}(\Om)]^{m}$. A sufficient condition 
is given in the next Lemma.

\begin{lemma}\label{lem:lemmino}
If the function $t\vf'(t)/\vf(t)$ is bounded on $(0,+\infty)$, then inequality \eqref{eq:invdis} holds.
\end{lemma}
\begin{proof}
Assuming $|t\vf'(t)/\vf(t)| \leq K$ ($t>0$), \eqref{eq:dircomp} implies
\begin{gather*}
|\nabla(\sqrt{\vf(|u|)}\, u)|^{2} \leq(K^{2}/4 + K) \, \vf(|u|)\,  |\nabla |u| |^{2}  + \vf(|u|)\,  | \nabla u|^2
\leq \\
(K^{2}/4 + K+1)\, \vf(|u|)\,  | \nabla u|^2\, .
\end{gather*}
\end{proof}

\begin{remark}\label{rem:lemmino}
Lemma \ref{lem:lemmino} and inequality \eqref{eq:lastineq} show that if $t\vf'(t)/\vf(t)$ is bounded, then
$|\nabla(\sqrt{\vf(|u|)}\, u)|^{2}$ and $\vf(|u|) \,  | \nabla u|^2$ are equivalent.
\end{remark}
%Here and in the sequel we  adopt the 
%summation convention and we put $p\in(1,\infty)$, $p'=p/(p-1)$.
%By $\Cspt^{1}(\Om)$ we denote the space of all the $C^{1}$
%functions having 
%compact support in $\Om$.

The next results of this section extend some of the results obtained
 in \cite{CM2006} in the case of $L^p$-dissipativity, i.e. when $\vf(t)=t^{p-2}$.
 If $\A$ is a matrix, by $\A^*$ we denote the adjoint matrix of $\A$, i.e.
 $\A^*=\overline{\A}^t$, $\A^t$ being the transposed matrix of $\A$.

\begin{lemma}\label{le:lemma1}
Let $\Om$ be  a domain in $\R^N$. The operator \eqref{eq:A} is $L^{\Phi}$-dissipative if and only if
\begin{equation}\label{eq:cond1}
\begin{gathered}
\Re \int_{\Om} \Big(\lan \A^{hk} \de_k	 v, \de_h v\ran 
+\LA(|v|)\, |v|^{-2} \lan \left(\A^{hk}- (\A^{kh})^*\right)v, \de_h v \ran \Re \lan v, \de_{k} v\ran  \\
- \LA^{2}(|v|)\, |v|^{-4} \lan \A^{hk} v,v\ran \Re \lan v, \de_{k} v\ran \Re \lan v, \de_{h} v\ran
\Big) dx  \geq 0
\end{gathered}
\end{equation}
for any $v \in [\Hspt^{1}(\Om)]^m$. Here and in the sequel the integrand
is extended by zero on the set where $v$ vanishes.
\end{lemma}
\begin{proof}
\textit{Sufficiency.}
	Suppose $r\geq 0$. Let $u\in [\Hspt^{1}(\Om)]^m$ such that $\vf(|u|)\, u\in [\Hspt^{1}(\Om)]^m$ and define $v=\sqrt{\vf(|u|)}\, u$. In view
	of   Lemma \ref{lemma:fipsi} we have that $v$ belongs to 
	$[\Hspt^{1}(\Om)]^m$.
	
	Since $|u|=\ze(|v|)$ and $|v|^{-1}v = |u|^{-1}u$, we get $u=|v|^{-1}v\, \ze(|v|)= \Th(|v|)\, v$ (see \eqref{eq:defHGA}). 
	Moreover from $\vf(|u|)=|u|^{-2}|v|^{2}=[\Th(|v|)]^{-2}$ we deduce
	 $\vf(|u|)\, \overline{u} = [\Theta(|v|)]^{-1}\overline{v}$.  Therefore
\begin{gather*}
\lan \A^{hk} \de_{k} u, \de_{h}(\vf(|u|)\, u)\ran =
	\lan \A^{hk} \de_{k}(\Theta(|v|)\, v),\de_{h}([\Theta(|v|)]^{-1}v)\ran \\
	= \lan \A^{hk}(\Theta'(|v|) v \de_{k}|v| + \Theta(|v|) \de_{k} v,
	- \Theta'(|v|) [\Theta(|v|)]^{-2} v\de_{h}|v|  + [\Theta(|v|)]^{-1} \de_{h} 
	v \ran  \\
	 = -(\Theta'(|v|)[\Theta(|v|)]^{-1})^{2}\lan \A^{hk}v, v\ran \, \de_{k} |v| \, \de_{h} |v|
	 \\ +
	 \Theta'(|v|)[\Theta(|v|)]^{-1}(\lan \A^{hk}v,\de_{h} v\ran \de_{k}|v| 
	 - \lan \A^{hk}\de_{k} v, v\ran \de_{h}|v| )
	  + 
	 \lan \A^{hk} \de_{k} v, \de_{h} v\ran .
\end{gather*}

From the identities
\begin{equation}\label{eq:identities}
\begin{gathered}
\de_{k}|v|=|v|^{-1}\Re \lan v, \de_{k} v\ran\, ,\\
 \lan \A^{hk}\de_{k} v, v\ran\, \de_{h}|v| = 
  \lan \A^{kh}\de_{h} v, v\ran\, \de_{k}|v| =
  \overline{ \lan (\A^{kh})^{*} v, \de_{h} v\ran}\,  \de_{k}|v| \, ,
\end{gathered}
\end{equation}
it follows
\begin{gather*}
\Re \lan \A^{hk} \de_{k} u, \de_{h}(\vf(|u|)\, u)\ran =\Re \big(
\lan \A^{hk} \de_{k} v, \de_{h} v\ran \\
+ \LA(|v|) |v|^{-2} 
\lan (\A^{hk}-(\A^{kh})^{*})v, \de_{h} v\ran \Re \lan v, \de_{k} v\ran \\
	 -\LA^{2}(|v|)|v|^{-4} \lan \A^{hk}v, v\ran \, \Re \lan v, \de_{h} v\ran
	 \Re \lan v, \de_{k} v\ran \big)
\end{gather*}
on the set $\{x\in \Om\ |\ u(x)\neq 0\}=\{x\in \Om\ |\ v(x)\neq 0\}$.

Inequality \eqref{eq:cond1} implies
$$
\Re \int_{\Om} \lan \A^{hk} \de_{k} u, \de_{h}(\vf(|u|)\, u)\ran\, dx \geq 0
$$
and the sufficiency is proved when $r\geq 0$.

If $-1<r<0$, setting $w=\vf(|u|)\, u$, i.e. $u=\psi(|w|)\, w$, we  can write  condition \eqref{eq:defdiss0} as
$$
\Re \int_{\Om} \lan (\A^{kh})^{*} \de_{k} w, \de_{h}(\psi(|w|)\, w)\ran\, dx \geq 0
$$
for any $w\in [\Hspt^{1}(\Om)]^{m}$ such that $\psi(|w|)\, w\in [\Hspt^{1}(\Om)]^{m}$.

Recalling Lemma \ref{lemma:new}, what we have already proved for $r\geq 0$ shows that this inequality holds 
 if
 \begin{equation}\label{eq:condvstar}
\begin{gathered}
\Re \int_{\Om} \Big(\lan (\A^{kh})^{*} \de_k	 v, \de_h v\ran 
+\widetilde{\LA}(|v|)\, |v|^{-2} \lan \left((\A^{kh})^{*}- \A^{hk}\right)v, \de_h v \ran \Re \lan v, \de_{k} v\ran  \\
-\widetilde{\LA}^{2}(|v|)\, |v|^{-4} \lan (\A^{kh})^{*} v,v\ran \Re \lan v, \de_{k} v\ran \Re \lan v, \de_{h} v\ran
\Big) dx  \geq 0
\end{gathered}
\end{equation}
for any $v \in [\Hspt^{1}(\Om)]^m$.
Since $\widetilde{\Lambda}(|v|)=-\Lambda(|v|)$ (see \eqref{eq:GAtilde}), conditions \eqref{eq:condvstar}
coincides with \eqref{eq:cond1} and the sufficiency is proved also for $-1<r<0$.

\textit{Necessity.}
Let $v\in [\Cspt^{1}(\Om)]^m$ and define $u_{\ep}=\Theta(g_{\ep})\, v$, where
$g_{\ep}=\sqrt{|v|^{2}+\ep^{2}}$.

The function $u_{\ep}$ and  $\vf(|u_{\ep}|)\, u_{\ep}$ belong to $[\Cspt^{1}(\Om)]^m$ and we have
\begin{gather*}
\lan \A^{hk} \de_k u_{\ep},\de_h (\vf(|u_{\ep}|)\, u_{\ep}\ran \\ =
\vf(|u_{\ep}|)\, \lan\A^{hk} \de_k u_{\ep},\de_h u_{\ep}\ran +
\vf'(|u_{\ep}|)\lan \A^{hk} \de_k u_{\ep},  u_{\ep}\,\de_h (|u_{\ep}|) \ran
\\ =
\vf[\Theta(g_{\ep})\, |v|]\,
\lan \A^{hk} (\Theta'(g_{\ep})\, v\, \de_k  g_{\ep} + \Theta(g_{\ep})\de_k  v) ,
\Theta'(g_{\ep})\, v\, \de_h  g_{\ep} + \Theta(g_{\ep})\de_h v \ran \\ +
\vf'[\Theta(g_{\ep})\, |v|] \\ \times
\lan \A^{hk} (\Theta'(g_{\ep})\, v\, \de_k g_{\ep} + \Theta(g_{\ep})\de_k v) ,
\Theta(g_{\ep})\, v\, 
(\Theta'(g_{\ep})\, |v|\, \de_h  g_{\ep} + \Theta(g_{\ep})\de_h |v|) \ran \, .
\end{gather*}

By expanding the terms in the last expression, we get
\begin{equation}
\begin{gathered}\label{eq:formulona}
 \lan \A^{hk} \de_k u_{\ep},\de_h (\vf(|u_{\ep}|)\, u_{\ep}\ran \\ =
\vf[\Theta(g_{\ep})\, |v|]\big\{ [\Theta'(g_{\ep})]^{2}
\lan \A^{hk} v,v\ran \, \de_k g_{\ep}\, \de_h g_{\ep}  \\
+ \Theta'(g_{\ep})\, \Theta(g_{\ep})\, [ \lan \A^{hk} v, \de_h v \ran
\, \de_k g_{\ep}  +
\lan \A^{hk}  \de_k v, v  \ran\, \de_h g_{\ep}] +
 \Theta^{2}(g_{\ep})\, \lan \A^{hk} \de_k v, \de_h v \ran \big\}  \\
+\vf'[\Theta(g_{\ep})\, |v|] \big\{
\Theta(g_{\ep})[\Theta'(g_{\ep})]^{2}|v| \lan \A^{hk} v,v\ran\, \de_k g_{\ep}\, \de_h
g_{\ep}  \\ +
\Theta^{2}(g_{\ep})\, \Theta'(g_{\ep}) [ \lan \A^{hk} v,v\ran\,\de_k g_{\ep}\, \de_h
|v|  + |v|  \lan \A^{hk} \de_k v,v\ran\, \de_h g_{\ep}  ]  \\ +
\Theta^{3}(g_{\ep})  \lan \A^{hk} \de_k v,v\ran\, \de_h |v|  \big\}
\, .
\end{gathered}
\end{equation}

Letting $\ep \to 0^{+}$ the right hand side tends to
\begin{gather*}
\vf[\Theta(|v|)\, |v|]\Big\{ [\Theta'(|v|)]^{2}
\lan \A^{hk} v,v\ran \, \de_k |v| \de_h |v|  \\ +
\Theta'(|v|) \Theta(|v|) [ \lan \A^{hk} v, \de_h v \ran
 \de_k |v|  +
\lan \A^{hk}  \de_k v, v  \ran \de_h |v|] \\ +
 \Theta^{2}(|v|) \lan \A^{hk} \de_k v, \de_h v \ran \Big\}  \\ +
\vf'[\Theta(|v|)\, |v|] \Big\{
\Theta(|v|)[\Theta'(|v|)]^{2}|v| \lan \A^{hk} v,v\ran\, \de_k |v|\, \de_h
|v| \\ + 
\Theta^{2}(|v|)\, \Theta'(|v|) [ \lan \A^{hk} v,v\ran\,\de_k |v|\, \de_h
|v|  + |v|  \lan \A^{hk} \de_k v,v\ran\, \de_h |v|  ]  \\  + 
\Theta^{3}(|v|)  \lan \A^{hk} \de_k v,v\ran\, \de_h |v|  \Big\}
\end{gather*}
Collecting similar terms gives 
\begin{equation}
\begin{gathered}\label{eq:rhstt}
\vf[\Theta(|v|)\, |v|]\, \Theta^{2}(|v|)\, \lan \A^{hk} \de_k v, \de_h v\ran
 \\  +
\vf[\Theta(|v|)\, |v|]\, \Theta'(|v|)\, \Theta(|v|)\, \lan \A^{hk} v,\de_h v \ran \de_k |v|
 \\   +
\Theta(|v|)\big\{ \vf[\Theta(|v|)\, |v|]\,\Theta'(|v|)  \\  +
\vf'[\Theta(|v|)\, |v|]\, 
\Theta(|v|)\,  [\Theta'(|v|)\, |v| +
\Theta(|v|)]\big\} \lan \A^{hk} \de_k v, v\ran \, \de_h |v| \\   + 
\Theta'(|v|) \big\{ \vf[\Theta(|v|)\, |v|]\,\Theta'(|v|) \\   +
\vf'[\Theta(|v|)\, |v|]\, \Theta(|v|)\, [\Theta'(|v|)\, |v|+ \Theta(|v|)]\big\}
\lan \A^{hk} v,v\ran\, \de_k  |v|\, \de_h |v| 
\end{gathered}
\end{equation}
on the set $\Om_{0}=\{x\in \Om\ |\ v(x)\neq 0\}$.
In view of  \eqref{eq:H2psi} and \eqref{eq:tH'} we have
\begin{gather*}
\vf[\Theta(|v|)\, |v|]\, \Theta^{2}(|v|) = 1, 
\ \vf[\Theta(|v|)\, |v|]\, \Theta'(|v|)\, \Theta(|v|) =  \Theta'(|v|)/\Theta(|v|), \\
\vf[\Theta(|v|)\, |v|]\,\Theta'(|v|) +
\vf'[\Theta(|v|)\, |v|]\, \Theta(|v|)\, [\Theta'(|v|)\, |v|+ \Theta(|v|)] \\   =
-\Theta'(|v|)/\Theta^{2}(|v|).
\end{gather*}

Substituting these equalities in \eqref{eq:rhstt} and keeping in mind \eqref{eq:formulona}, we find
\begin{gather*}
\lim_{\ep\to 0^{+}}\lan \A^{hk} \de_k u_{\ep},\de_h (\vf(|u_{\ep}|)\, u_{\ep}\ran \\  =
 \lan \A^{hk} \de_k v, \de_h v\ran + 
 \Theta'(|v|)/\Theta(|v|)\, [\lan \A^{hk} v, \de_h v\ran \,  \de_k |v| -
 \lan \A^{hk} \de_k v,  v\ran \,  \de_h |v| ]\, .\\
  - \left(\Theta'(|v|)/\Theta(|v|)\right)^2  \lan \A^{hk}  v,  v\ran \, \de_k |v|\,  \de_h |v|\, .
 \end{gather*}
 on $\Om_{0}$. Thanks to  \eqref{eq:identities} this is the same as
 \begin{gather*}
\lan \A^{hk} \de_k v, \de_h v\ran 
+\LA(|v|)\, |v|^{-2}  \left(\lan \A^{hk}  v, \de_h v \ran  -
\overline{\lan (\A^{kh})^* v, \de_h v \ran} \right) \Re \lan v, \de_{k} v\ran  \\
- \LA^{2}(|v|)\, |v|^{-4} \lan \A^{hk} v,v\ran \Re \lan v, \de_{k} v\ran \Re \lan v, \de_{h} v\ran\, .
\end{gather*}

As in \cite[(58)]{CM2021}
one can prove that each term in the last expression of \eqref{eq:formulona} can be
majorized by $L^{1}$ functions not depending on $\ep$. 
By the Lebesgue dominated convergence theorem, we get
\begin{equation}\label{eq:weget}
\begin{gathered}
\lim_{\ep\to 0^{+}} \Re \int_{\Om}\lan \A^{hk} \de_k u_{\ep},\de_h (\vf(|u_{\ep}|)\, u_{\ep}\ran dx =
\Re \int_{\Om} \Big(\lan \A^{hk} \de_k  v, \de_h v\ran  \\
+\LA(|v|)\, |v|^{-2} \lan \left(\A^{hk}- (\A^{kh})^*\right)v, \de_h v \ran \Re \lan v, \de_{k} v\ran  \\
- \LA^{2}(|v|)\, |v|^{-4} \lan \A^{hk} v,v\ran \Re \lan v, \de_{k} v\ran \Re \lan v, \de_{h} v\ran
\Big) dx  \, .
\end{gathered}
\end{equation}

The left hand side 
being non negative (see \eqref{eq:defdiss0}), inequality
\eqref{eq:cond1} holds for any $v\in [\Cspt^{1}(\Om)]^{m}$.

Let  now $v\in [\Hspt^{1}(\Om)]^{m}$ and 
$v_{n}\in [\Cspt^{\infty}(\Om)]^{m}$  such that
$v_{n}\to v$ almost everywhere
in $\Om$ and in $H^{1}$ norm. 
Reasoning as  in \cite[Lemma 5]{CM2021} 
one can prove that
\begin{gather*}
\lim_{n\to\infty}
\Re \int_{\Om} \Big(\lan \A^{hk} \de_k	 v_{n}, \de_h v_{n}\ran 
\\ + 
\LA(|v_{n}|)\, |v_{n}|^{-2} \lan \left(\A^{hk}- (\A^{kh})^*\right)v_{n}, \de_h v_{n} \ran \Re \lan v_{n}, \de_{k} v_{n}\ran  \\
- \LA^{2}(|v_{n}|)\, |v_{n}|^{-4} \lan \A^{hk} v_{n},v_{n}\ran \Re \lan v_{n}, \de_{k} v_{n}\ran \Re \lan v_{n}, \de_{h} v_{n}\ran
\Big) dx\\  =
\Re \int_{\Om} \Big(\lan \A^{hk} \de_k	 v, \de_h v\ran 
+\LA(|v|)\, |v|^{-2} \lan \left(\A^{hk}- (\A^{kh})^*\right)v, \de_h v \ran \Re \lan v, \de_{k} v\ran  \\
- \LA^{2}(|v|)\, |v|^{-4} \lan \A^{hk} v,v\ran \Re \lan v, \de_{k} v\ran \Re \lan v, \de_{h} v\ran
\Big) dx \, ,
\end{gather*}
and the result follows.
\end{proof}

We have also
\begin{lemma}\label{le:7}
Let $\Om$ be  a domain in $\R^N$. The operator \eqref{eq:A} is strict  $L^{\Phi}$-dissipative if and only if
there exists $\kappa>0$ such that
\begin{equation}\label{eq:cond1str}
\begin{gathered}
\Re \int_{\Om} \Big(\lan \A^{hk} \de_k	 v, \de_h v\ran 
 +
\LA(|v|)\, |v|^{-2} \lan \left(\A^{hk}- (\A^{kh})^*\right)v, \de_h v \ran \Re \lan v, \de_{k} v\ran  \\
- \LA^{2}(|v|)\, |v|^{-4} \lan \A^{hk} v,v\ran \Re \lan v, \de_{k} v\ran \Re \lan v, \de_{h} v\ran
\Big) dx  \geq
 \kappa \int_{\Om}  |\nabla v|^2 dx
\end{gathered}
\end{equation}
for any $v \in [\Hspt^{1}(\Om)]^m$. 
\end{lemma}
\begin{proof}
\textit{Sufficiency.} As in the proof of the previous Lemma, 
	suppose $r\geq 0$ and take  $v=\sqrt{\vf(|u|)}\, u$, where $u\in [\Hspt^{1}(\Om)]^m$ is such that $\vf(|u|)\, u\in [\Hspt^{1}(\Om)]^m$.
	In Lemma \ref{le:lemma1} we showed that the left hand side of \eqref{eq:defdiss0str} coincides
	with the left hand side of \eqref{eq:cond1str}.  The right hand sides  being  equal, the sufficiency is proved when $r\geq 0$.
	
	If $-1<r<0$, 
	setting $w=\vf(|u|)\, u$, i.e. $u=\psi(|w|)\, w$,  and recalling \eqref{eq:phipsi}, we  can write  condition \eqref{eq:defdiss0str} as
$$
\Re \int_{\Om} \lan (\A^{kh})^{*} \de_{k} w, \de_{h}(\psi(|w|)\, w)\ran\, dx \geq \kappa \int_{\Om} | \nabla(\sqrt{\psi(|w|)}\, w)|^2 dx
$$
for any $w\in [\Hspt^{1}(\Om)]^{m}$ such that $\psi(|w|)\, w\in [\Hspt^{1}(\Om)]^{m}$.
As in the previous Lemma this implies \eqref{eq:cond1str} for any $v\in [\Hspt^{1}(\Om)]^{m}$.

	\textit{Necessity.} As in the proof of Necessity in Lemma \ref{le:lemma1}, let $v\in [\Cspt^{1}(\Om)]^m$ and define $u_{\ep}=\Theta(g_{\ep})\, v$.
	 Let us consider the integral
	$$
	\int_{\Om} | \nabla(\sqrt{\vf(|u_{\ep}|)}\, u_{\ep})|^2 dx\, .
	$$
	
	Let us write $\sqrt{\vf(|u_{\ep}|)}\, u_{\ep}$ as 
	$\ro_{\ep}v$, where $\ro_{\ep}=\sqrt{\vf(|u_{\ep}|)}\, \Theta(g_{\ep})$. Keeping in mind \eqref{eq:H2psi}, we find
	$$
	\lim_{\ep\to 0^{+}}\ro_{\ep}=\lim_{\ep\to 0^{+}}\sqrt{\vf[\Theta(g_{\ep}) |v|]}\, \Theta(g_{\ep})=
	\sqrt{\vf[\Theta(|v|)  |v|]}\, \Theta(|v|)=1
	$$
	on $\Om_{0}=\{x\in \Om\ |\ v(x)\neq 0\}$. Moreover
	\begin{gather*}
\lim_{\ep\to 0^{+}} \de_{h} \ro_{\ep} = \lim_{\ep\to 0^{+}} \Big(\frac{1}{2}\, \frac{\vf'[\Theta(g_{\ep}) |v|]}{\sqrt{\vf[\Theta(g_{\ep}) |v|]}}\, (|v|\, \Theta'(g_{\ep})\,  \de_{h}g_{\ep} + \Theta(g_{\ep})\, \de_{h}|v|)
\Theta(g_{\ep})  \\
+ \sqrt{\vf[\Theta(g_{\ep}\, |v|]}\, \Theta'(g_{\ep})\, \de_{h}g_{\ep}\Big) \\ =
\frac{\vf'[\Theta(|v|) |v|]\, \Theta(|v|)[\Theta'(|v|) |v| + \Theta(|v|)] + 2 \vf[\Theta(|v|) |v|]  \, \Theta'(|v|)}{2 \sqrt{\vf[\Theta(|v|) |v|]}}\, .
\end{gather*}

Equality 
\eqref{eq:tH'} shows that the numerator in the last expression can be written as
\begin{gather*}
-\Theta'(|v|)/\Theta^{2}(|v|)\,  + \vf[\Theta(|v|) |v|]  \, \Theta'(|v|) \\ = 
\Theta'(|v|)( -1 + \vf[\Theta(|v|) |v|] \, \Theta^{2}(|v|)) /\Theta^{2}(|v|) = 0
\end{gather*}
(see also \eqref{eq:H2psi}) and then 
$$
\lim_{\ep\to 0^{+}} \de_{h} \ro_{\ep} = 0
$$
on $\Om_{0}$. This implies that
$$
\lim_{\ep\to 0^{+}} \de_{h}(\sqrt{\vf(|u_{\ep}|)}\, u_{\ep}) = 
\lim_{\ep\to 0^{+}} \de_{h}(\ro_{\ep} v) =
\lim_{\ep\to 0^{+}}(v \de_{h}\ro_{\ep} + \ro_{\ep} \de_{h} v) = \de_{h}v
$$
on $\Om_{0}$.

By Fatou's Lemma we get
$$
\int_{\Om}  |\nabla v|^2 dx \leq \liminf_{\ep\to 0^+}\int_{\Om} |\nabla( \sqrt{\vf(|u_{\ep}|)}\, u_{\ep})|^2 dx\, .
$$

	On the other hand we know  that \eqref{eq:weget} holds and therefore the inequality
	$$
	\Re \int_{\Om} \lan \A^{hk} \de_{k} u_{\ep}, \de_{h}(\vf(|u_{\ep}|)\, u_{\ep})\ran\, dx \geq 
\kappa \int_{\Om} | \nabla(\sqrt{\vf(|u_{\ep}|)}\, u_{\ep})|^2 dx
	$$
	implies \eqref{eq:cond1str} for any $v\in [\Cspt^{1}(\Om)]^m$.
	
	The result for any $v\in [\Hspt^{1}(\Om)]^{m}$ follows by approximating  $v$ by a sequence
$v_{n}\in [\Cspt^{\infty}(\Om)]^{m}$ (as in the previous Lemma).
\end{proof}

We conclude this Section with the following Corollary concerning the strict  $L^{\Phi}$-dissipativity
of the operator \eqref{eq:A}.

\begin{corollary}\label{co:co1}
 Suppose
 \begin{equation}\label{eq:supL}
\sup_{t>0}\LA^{2}(t)<1\, .
\end{equation}
The operator $A$ is strict  $L^{\Phi}$-dissipative if and only if there exists $\kappa>0$ such that
$A-\kappa \Delta$ is $L^{\Phi}$-dissipative.
\end{corollary}
\begin{proof}
 If the operator $A$ is strict  $L^{\Phi}$-dissipative, \eqref{eq:cond1str} holds.  This implies
\begin{equation}\label{eq:A-kD}
 \begin{gathered}
\Re \int_{\Om} \Big(\lan \A^{hk} \de_k	 v, \de_h v\ran 
\\ +
\LA(|v|)\, |v|^{-2} \lan \left(\A^{hk}- (\A^{kh})^*\right)v, \de_h v \ran \Re \lan v, \de_{k} v\ran  \\
- \LA^{2}(|v|)\, |v|^{-4} \lan \A^{hk} v,v\ran \Re \lan v, \de_{k} v\ran \Re \lan v, \de_{h} v\ran
\Big) dx  \geq\\
 \kappa \int_{\Om}\Big(  |\nabla v|^2 - \LA^{2}(|v|)\,  |\nabla|v||^2\Big) dx
\end{gathered}
\end{equation}
for any $v \in [\Hspt^{1}(\Om)]^m$.  Observing that if $\{a^{hk}_{ij}\}=\{\delta_{hk}\delta_{ij}\}$ we have
$$
|v|^{-4} \lan \A^{hk} v,v\ran \Re \lan v, \de_{k} v\ran \Re \lan v, \de_{h} v\ran = |v|^{-2} \Re \lan v, \de_{k} v\ran \Re \lan v, \de_{k}v\ran =
|\nabla|v||^{2},
$$
Lemma \ref{le:lemma1} shows that $A-\kappa \Delta$ is $L^{\Phi}$-dissipative.

Viceversa, if $A-\kappa \Delta$ is $L^{\Phi}$-dissipative, inequality \eqref{eq:A-kD} holds for any $v \in [\Hspt^{1}(\Om)]^m$. 
Since 
\begin{equation}\label{eq:nablav2}
|\nabla|v||^{2}\leq   |\nabla v|^2,
\end{equation}
we find
$$
 \int_{\Om}\Big(  |\nabla v|^2 - \LA^{2}(|v|)\,  |\nabla|v||^2\Big) dx \geq
 (1-\sup_{t>0}\LA^{2}(t)) \int_{\Om}  |\nabla v|^2 dx\, .
$$

Thanks to \eqref{eq:supL},  the inequality \eqref{eq:cond1str} holds with $\kappa$ replaced by the positive constant $\kappa(1-\sup_{t>0}\LA^{2}(t))$.
\end{proof}	

We remark that in the case of $L^p$-dissipativity, i.e. $\vf(t)=t^{p-2}$ ($1<p<\infty$), condition \eqref{eq:supL} is satisfied, because
$\LA^{2}(t)=(1-2/p)^{2}$.

\section{A necessary condition for the functional dissipativity when $N=2$}\label{sec:CN}

From now on we require  also the following condition on the function $\vf$:
\begin{enumerate}\setcounter{enumi}{5}
\item\label{item6}  the function
$$
|s\, \vf'(s)/ \vf(s)|
$$
is not decreasing.
\end{enumerate}

A first consequence of  condition \ref{item6} is the following

\begin{lemma}
The function $\LA^2(t)$ is not decreasing on $(0,+\infty)$.
\end{lemma}
\begin{proof}
Since the function $s\sqrt{\vf(s)}$ is stricly increasing and its range is $(0,+\infty)$, the function $\LA^2(t)$ 
is not decreasing if and only if the function $\LA^2(s\sqrt{\vf(s)})$ is not decreasing.
Define $\ga(s)= s\,\vf'(s)/\vf(s)$. Note that the monotoniciy of $|\ga(s)|$ (see condition \ref{item6}) implies that if there exists $\si>0$ such that $\ga(\si)=0$,
then $\ga(s)=0$ for any $0<s\leq \si$.  If there exists $s>0$ such that $\ga(s)=0$ put 
$$
\si_{0}=\inf\{ \si>0\ |\ \ga(\si)=0\}\, .
$$
If $\ga(s)\neq 0$ for any $s>0$, put $\si_{0}=0$.
Note that, in any case, $\ga(s)\neq 0$ for any $s>\si_{0}$.
Recalling \eqref{eq:G=}, we have that $\LA^{2}(s\sqrt{\vf(s)})=0$ for any $s\leq \si_0$  and $\LA^{2}(s\sqrt{\vf(s)})>0$ for any $s> \si_0$.
Then it is  it suffices to prove that  $\LA^{2}(s\sqrt{\vf(s)})\leq  \LA^{2}(t\sqrt{\vf(t)})$ for any $ \si_{0} < s<t$.
Defining 
$$
\ga(s)= \frac{s \vf'(s)}{s \vf'(s) + 2\vf(s)}
$$ 
and observing that condition \ref{item2} implies $\ga(s)+2>0$ (for any $s>0$), we have that $\LA^{2}(s\sqrt{\vf(s)})\leq  \LA^{2}(t\sqrt{\vf(t)})$ means
$$
|\ga(s)|\, (\ga(t) + 2) \leq |\ga(t)|\, (\ga(s) + 2) \, .
$$

Since $|\ga(s)|\, \ga(t)=|\ga(t)|\, \ga(s)$, the last inequality is equivalent to $|\ga(s)| \leq |\ga(t)|$, which is true in view of condition \ref{item6}.
\end{proof}

We remark that, since the function $\LA$ does not change the sign, the previous result implies the monotonicity of the bounded function
 $\LA(s)$ and then  the existence of the finite limit
  \begin{equation}\label{Lam}
\LA_{\infty}=\lim_{t\to +\infty} \LA(t) \, .
\end{equation}

We have also
  \begin{equation}\label{Lamsup}
\LA^{2}_{\infty}=\sup_{t>0} \LA^{2}(t) \, .
\end{equation}

The next theorem provides a necessary condition for the $L^{\Phi}$-dissipativity of operator $A$ when $N=2$.

\begin{theorem}\label{th:algcond}
    Let $\Om$ be a domain of $\R^2$. If the operator \eqref{eq:A} is $L^{\Phi}$-dissipative we have
       \begin{equation}\label{eq:algcond}
       \begin{gathered}
   \Re  \Big( \lan (\A^{hk}(x)\xi_{h}\xi_{k})\h,\h\ran  -\LA^{2}_{\infty}\, \lan (\A^{hk}(x)\xi_{h}\xi_{k})\om,\om\ran (\Re \lan 
    \h,\om\ran)^{2}\\
    +\LA_{\infty}\, (\lan (\A^{hk}(x)\xi_{h}\xi_{k})\om,\h\ran
	-\lan (\A^{hk}(x)\xi_{h}\xi_{k}) \h,\om\ran) 
	\Re\lan \h,\om\ran \Big)
	\geq 0
	\end{gathered}
	\end{equation}
	for almost every $x\in\Om$ and 
	for any $\xi\in\R^{2}$, $\h,\, \om\in \C^{m}$, $|\om|=1$. Here the constant $\LA_{\infty}$ is given by 
	\eqref{Lam}.
\end{theorem}
\begin{proof}
As in \cite[Theorem 2]{CM2006}, let us assume first that $\A^{hk}$ are constant matrices and that $\Om=\R^2$.
Let us fix $\om\in\C^{m}$ with $|\om|=1$ and take
   $
   v(x)=w(x)\, g(\log|x|/\log R)
   $,
   where 
   \begin{equation*}
       w(x)=\mu\om+\p(x),
    %   \label{defw}
   \end{equation*}
   $\mu,\,R\in\R^{+}$, $R>1$,
   $\p\in(\Cspt^{\infty}(\R^{2}))^{m}$, 
   $g\in C^{\infty}(\R)$, $g(t)=1$ if $t\leq 1/2$ and
   $g(t)=0$ if $t\geq 1$.

Put the function $v$ in \eqref{eq:cond1} and let $R\to +\infty$. Using the same arguments as in the first part of the proof of \cite[Theorem 2]{CM2006}
and observing that $\LA$ is continuous and $|\LA(t)|<1$ (see \cite[(32)]{CM2021}), we find
\begin{equation}\label{eq:cond1w}
\begin{gathered}
\Re \int_{B_{\del}(0)}\!\! \big(\lan \A^{hk} \de_k	 w, \de_h w\ran 
\\ 
- \LA^{2}(|w|)\, |w|^{-4} \lan \A^{hk} w,w\ran \Re \lan w, \de_{k} w\ran \Re \lan w, \de_{h} w\ran \\
+\LA(|w|)\, |w|^{-2} \lan \left(\A^{hk}- (\A^{kh})^*\right)w, \de_h w \ran \Re \lan w, \de_{k} w\ran 
\big) dx  \geq 0,
\end{gathered}
\end{equation}
where $\del>0$ is such that $\spt \p \subset B_{\del}(0)$.

The first term in \eqref{eq:cond1w} is 
$$
\Re \lan \A^{hk}\de_{k}w,\de_{h}w\ran =
   \Re \lan \A^{hk}\de_{k}\p,\de_{h}\p \ran ,
$$
while the second one can be written as
\begin{gather*}
      \LA^{2}(|w|) |w|^{-4}\Re\lan \A^{hk}w,w\ran \Re \lan 
       w,\de_{k}w\ran \Re \lan w,\de_{h}w\ran \\ =
       \LA^{2}(|\mu\om+\p|) |\mu\om+\p|^{-4}  \\   \times
       \Re\lan \A^{hk}(\mu\om+\p),\mu\om+\p\ran  \Re \lan 
	      \mu\om+\p,\de_{k}\p\ran  \Re \lan \mu\om+\p,\de_{h}\p\ran.
	      \end{gather*}	      
	     Finally the third one is equal to
	     \begin{gather*}
    \LA(|w|)\, |w|^{-2} \Re \lan \left(\A^{hk}- (\A^{kh})^*\right)w, \de_h w \ran \Re \lan w, \de_{k} w\ran  \\  
    =
     \LA(|\mu\om+\p|) |\mu\om+\p|^{-2}\Re(\lan (\A^{hk}- (\A^{kh})^*)(\mu\om+\p),\de_{h}\p\ran \Re \lan 
     \mu\om+\p,\de_{k}\p\ran .
\end{gather*}

Letting $\mu\to+\infty$ in \eqref{eq:cond1w}, we obtain
\begin{equation}\label{eq:cond1psi}
\begin{gathered}
\Re  \int_{\R^{2}}\big( \lan \A^{hk}\de_{k}\p,\de_{h}\p\ran
	  -\LA_{\infty}^{2} \lan \A^{hk}\om,\om\ran \Re \lan 
	  \om,\de_{k}\p\ran \Re \lan \om,\de_{h}\p\ran \\
	 + \LA_{\infty} ((\lan (\A^{hk}- (\A^{kh})^*)\om,\de_{h}\p\ran \Re \lan \om,\de_{k}\p\ran
  		  \big)\, dx  \geq 0. 
	  \end{gathered}
\end{equation}

Putting in \eqref{eq:cond1psi}
$$
\p(x)=\h\,\vf(x)\, e^{i\mu\lan\xi,x\ran}
$$
where $\h\in\C^{m}$, $\vf\in\Cspt^{\infty}(\R^{2})$ and $\mu$ is a real parameter,
by standard arguments (see, e.g., \cite[p.107--108]{ficheraLN}),  
we find \eqref{eq:algcond}.

If the matrices  $\A^{hk}$ are not constant and defined in $\Om$,  take
$$
v(x) = w((x-x_{0})/\ep)
$$
where $x_{0}\in\Om$ is a fixed point, $w\in [\Cspt^{\infty}(B_{1}(0))]^m$ and $0<\ep< \dist(x_{0},\de \Om)$.
Putting this particular $v$ in   \eqref{eq:cond1} we find
\begin{gather*}
	0\leq 
	\frac{1}{\ep^{2}}\Re \int_{\Om} \Big(\lan \A^{hk}(x) \de_k	 w((x-x_{0})/\ep), \de_h w((x-x_{0})/\ep)\ran 
 \\ +
\LA(|w((x-x_{0})/\ep)|) 
 |w((x-x_{0})/\ep)|^{-2} \\ \times
 \lan (\A^{hk}(x)- (\A^{kh})^*(x))w((x-x_{0})/\ep), \de_h w((x-x_{0})/\ep) \ran 
 \\  \times
\Re \lan w((x-x_{0})/\ep), \de_{k} w((x-x_{0})/\ep)\ran   \\ -
\LA^{2}(|w((x-x_{0})/\ep)|)\, |w((x-x_{0})/\ep)|^{-4} \lan \A^{hk}(x) w((x-x_{0})/\ep),w((x-x_{0})/\ep)\ran \\
 \times \Re \lan w((x-x_{0})/\ep), \de_{k} w((x-x_{0})/\ep)\ran \Re \lan w((x-x_{0})/\ep), \de_{h} w((x-x_{0})/\ep)\ran
\Big) dx 
\end{gather*}
and then
\begin{equation}\label{eq:doporeview}
\begin{gathered}
0\leq \Re \int_{\Om} \Big(\lan \A^{hk}(x_{0}+\ep y) \de_k w(y), \de_h w(y)\ran 
 \\  +
\LA(|w((y)|) 
 |w(y)|^{-2}
 \lan (\A^{hk}(x_{0}+\ep y)- (\A^{kh})^*(x_{0}+\ep y))w(y), \de_h w((y) \ran 
 \\  \times
\Re \lan w(y), \de_{k} w(y)\ran  - 
\LA^{2}(|w(y)|)\, |w(y)|^{-4} \lan \A^{hk}(x_{0}+\ep y) w(y),w(y)\ran \\
 \times
\Re \lan w(y), \de_{k} w(y)\ran \Re \lan w(y), \de_{h} w(y)\ran
\Big) dy   
\, .
\end{gathered}
\end{equation}
 Therefore 
 \begin{gather*}
\Re \int_{\Om} \Big(\lan \A^{hk}(x_{0}) \de_k w(y), \de_h w(y)\ran 
 \\ +
\LA(|w((y)|) 
 |w(y)|^{-2}
 \lan (\A^{hk}(x_{0})- (\A^{kh})^*(x_{0}y))w(y), \de_h w((y) \ran \\
 \times
\Re \lan w(y), \de_{k} w(y)\ran  - 
\LA^{2}(|w(y)|)\, |w(y)|^{-4} \lan \A^{hk}(x_{0}) w(y),w(y)\ran \\
 \times
\Re \lan w(y), \de_{k} w(y)\ran \Re \lan w(y), \de_{h} w(y)\ran
\Big) dy   \geq 0
\end{gather*}
almost everywhere, because this integral is the $\lim_{\ep\to 0^{+}}$ of \eqref{eq:doporeview} 
for almost any $x_{0}\in\Om$.
%\begin{gather*}
%	\lim_{\ep\to 0^{+}}\Re \int_{\Om} \Big(\lan \A^{hk}(x_{0}+\ep y) \de_k w(y), \de_h w(y)\ran 
% \\  +
%\LA(|w((y)|) 
% |w(y)|^{-2}
% \lan (\A^{hk}(x_{0}+\ep y)- (\A^{kh})^*(x_{0}+\ep y))w(y), \de_h w((y) \ran 
% \\  \times
%\Re \lan w(y), \de_{k} w(y)\ran  - 
%\LA^{2}(|w(y)|)\, |w(y)|^{-4} \lan \A^{hk}(x_{0}+\ep y) w(y),w(y)\ran \\
%\times
%\Re \lan w(y), \de_{k} w(y)\ran \Re \lan w(y), \de_{h} w(y)\ran
%\Big) dy   
% \geq 0
%\end{gather*}
% The arbitrariness of $w\in [\Cspt^{\infty}(\R^{2})]^m$ and what we have already obtained for 
%constant matrices give the result.
%\end{proof}
 The arbitrariness of $w\in [\Cspt^{\infty}(\R^{2})]^m$ and what we have already obtained for 
constant matrices give the result.
\end{proof}

With the same proof we have
\begin{theorem}\label{th:algcondstr}
    Let $\Om$ be a domain of $\R^2$. If the operator \eqref{eq:A} is strict  $L^{\Phi}$-dissipative, there  exists $\kappa>0$ such that
           \begin{gather*}
   \Re  \Big( \lan (\A^{hk}(x)\xi_{h}\xi_{k})\h,\h\ran  -\LA^{2}_{\infty}\, \lan (\A^{hk}(x)\xi_{h}\xi_{k})\om,\om\ran (\Re \lan 
    \h,\om\ran)^{2}\\
    +\LA_{\infty}\, (\lan (\A^{hk}(x)\xi_{h}\xi_{k})\om,\h\ran
	-\lan (\A^{hk}(x)\xi_{h}\xi_{k}) \h,\om\ran) 
	\Re\lan \h,\om\ran \Big)
	\geq \kappa\, |\xi|^{2} |\h|^{2}
	\end{gather*}
	for almost every $x\in\Om$ and 
	for any $\xi\in\R^{2}$, $\h,\, \om\in \C^{m}$, $|\om|=1$.
\end{theorem}

\section{Elasticity}\label{sec:elast}

In this section we consider the two-dimensional linear system of elasticity \eqref{opelast}.
%, i.e.
% the operator \eqref{eq:A} where
% $$
% a^{hk}_{ij} = \la \del_{ih}  \del_{jk}       + \mu   \del_{ij}  \del_{hk}   + \mu \del_{ik}  \del_{hj}
% $$
%$(1\leq i,j, h,k \leq 2)$.  
The Lam\'e coefficients $\la$, $\mu$ are supposed to be
measurable essentially bounded real valued functions such that
\begin{equation}\label{eq:lameinf}
\essinf_{x\in\Om} \mu(x)> 0\, ; \ \essinf_{x\in\Om} (\la(x)+2\mu(x))>0.
\end{equation}

The next Theorems provide necessary conditions for the $L^{\Phi}$-dissipativity and the strict $L^{\Phi}$-dissipativity  of  the elasticity operator.
\begin{theorem}\label{th:3}
If the operator \eqref{opelast} is $L^{\Phi}$-dissipative, then
    \begin{equation}\label{eq:condnuovaug}
\LA^{2}_{\infty} \leq   1 - \esssup_{x\in\Om}\left(\frac{\la + \mu}{\la+3\mu}\right)^2.
\end{equation}
\end{theorem}
\begin{proof}
 In view of Theorem \ref{th:algcond},  condition \eqref{eq:algcond} holds. We have
\begin{gather*}
   \lan (\A^{hk}\xi_{h}\xi_{k})\h,\h\ran =
   \mu\,  |\xi|^{2}|\h|^{2}+(\la+\mu) \lan\xi,\h\ran^{2}, \cr
    \lan (\A^{hk}\xi_{h}\xi_{k})\om,\om\ran =
	\mu\, |\xi|^{2}+(\la+\mu) \lan\xi,\om\ran^{2}
	%\lan (\A^{hk}\xi_{h}\xi_{k})\l,\o\ran =
	%    |\xi|^{2}|\lan \l,\o\ran +(1-2\n)^{-1}\lan\xi,\l\ran
	%    \lan\xi,\o\ran
\end{gather*}
    for any $\xi,\, \h,\, \om\in\R^{2}$, $|\om|=1$.  Since  $(\A^{kh})^{*}=\A^{hk}$,  condition \eqref{eq:algcond}
    can be written as
   \begin{equation}\label{eq:algcondel0}
 \mu\,  |\xi|^{2}|\h|^{2}+(\la+\mu) \lan\xi,\h\ran^{2} 
 -\LA^{2}_{\infty}[\mu\, |\xi|^{2}+(\la+\mu) \lan\xi,\om\ran^{2}]\lan \h, \om\ran^{2} \geq 0
\end{equation}
for almost any $x\in\Om$ and for any $\xi,\, \h,\, \om\in\R^{2}$, $|\om|=1$.  
Fix $x\in\Om$ and rewrite 
\eqref{eq:algcondel0} as
   \begin{equation}\label{eq:algcondel1}
  |\xi|^{2}|\h|^{2}+\mu^{-1}(\la+\mu) \lan\xi,\h\ran^{2} 
 -\LA^{2}_{\infty} [ |\xi|^{2}+\mu^{-1}(\la+\mu) \lan\xi,\om\ran^{2}]\lan \h, \om\ran^{2} \geq 0\, .
\end{equation}

Reasoning as in \cite[p.244--245]{CM2006} (just replace $(1-2\nu)^{-1}$ and $(1-2/p)$  in \cite{CM2006} by $\mu^{-1}(\la+\mu)$ and $\LA_{\infty}$, respectively)
we find that \eqref{eq:algcondel1}
implies
$$
\LA^{2}_{\infty} \leq   1 -\left(\frac{\la + \mu}{\la+3\mu}\right)^2.
$$
Taking the infimum of the right hand side, we get \eqref{eq:condnuovaug}.
\end{proof}

\begin{theorem}\label{th:4}
If the operator \eqref{opelast} is strict  $L^{\Phi}$-dissipative, then
    \begin{equation}\label{eq:condnuova}
\LA^{2}_{\infty} <   1 - \esssup_{x\in\Om}\left(\frac{\la + \mu}{\la+3\mu}\right)^2.
\end{equation}
\end{theorem}
\begin{proof}
Theorem \ref{th:algcondstr} shows that
$$
 \mu\,  |\xi|^{2}|\h|^{2}+(\la+\mu) \lan\xi,\h\ran^{2} 
 -\LA^{2}_{\infty}[\mu\, |\xi|^{2}+(\la+\mu) \lan\xi,\om\ran^{2}]\lan \h, \om\ran^{2} \geq \kappa\  |\xi|^{2}|\h|^{2}
 $$
for almost any $x\in\Om$ and for any $\xi,\, \h,\, \om\in\R^{2}$, $|\om|=1$.  

This implies that, for any $0<h\leq \kappa$ we have
$$
( \mu-h)\,  |\xi|^{2}|\h|^{2}+(\la+\mu) \lan\xi,\h\ran^{2} 
 -\LA^{2}_{\infty}[(\mu-h)\, |\xi|^{2}+(\la+\mu) \lan\xi,\om\ran^{2}]\lan \h, \om\ran^{2} \geq 0
 $$
for almost any $x\in\Om$ and for any $\xi,\, \h,\, \om\in\R^{2}$, $|\om|=1$. 

Setting $\widetilde{\mu}=\mu - h$, $\widetilde{\la}=\la+h$, we can rewrite the last inequalty as
$$
 \widetilde{\mu}\,  |\xi|^{2}|\h|^{2}+(\widetilde{\la}+\widetilde{\mu}) \lan\xi,\h\ran^{2} 
 -\LA^{2}_{\infty}[\widetilde{\mu}\, |\xi|^{2}+(\widetilde{\la}+\widetilde{\mu}) \lan\xi,\om\ran^{2}]\lan \h, \om\ran^{2} \geq 0\, .
$$

The proof of Theorem \ref{th:3} shows that this implies
$$
\LA^{2}_{\infty} \leq   1 - \esssup_{x\in\Om}\left(\frac{\widetilde{\la} + \widetilde{\mu}}{\widetilde{\la}+3\widetilde{\mu}}\right)^2 \, ,
$$
i.e.
\begin{equation}\label{eq:tilde}
\LA^{2}_{\infty} \leq   1 - \esssup_{x\in\Om}\left(\frac{\la + \mu}{\la+3\mu-2h}\right)^2\, .
\end{equation}

If $h$ is sufficiently small, we have
$$
\esssup_{x\in\Om}\left(\frac{\la + \mu}{\la+3\mu}\right)^2 < \esssup_{x\in\Om}\left(\frac{\la + \mu}{\la+3\mu-2h}\right)^2
$$
and then \eqref{eq:tilde} implies \eqref{eq:condnuova}.
\end{proof}

The next Theorem provides a sufficient  conditions for  the strict $L^{\Phi}$-dissipativity  of  the elasticity operator.

\begin{theorem}\label{th:5}
Assume that the $BMO$ seminorm of the function $ \mu^2\, (\la + 3\mu)^{-1}$ is sufficiently small.
     If \eqref{eq:condnuova} holds, then
the elasticity operator \eqref{opelast} is strict  $L^{\Phi}$-dissipative.
\end{theorem}
\begin{proof}
We note that condition \eqref{eq:condnuova} implies $\LA^{2}_{\infty} <   1$. In view of  \eqref{Lamsup}  and Corollary \eqref{co:co1},
if we prove that there exists $h>0$ such that $E-h\Delta$ is $L^{\Phi}$-dissipative, the assertion follows.

Let $\delta$ be a real constant such that
\begin{equation}\label{eq:delta}
0<\delta <   1 - \esssup_{x\in\Om}\left(\frac{\la + \mu}{\la+3\mu}\right)^2 - \LA^{2}_{\infty}
\end{equation}
and, taking into account \eqref{eq:lameinf}, choose $\kappa$ such that
\begin{equation}\label{eq:defkappa}
0<\kappa < \frac{\delta}{2(1-\LA^{2}_{\infty})}\, \min \Big\{\essinf_{x\in\Om}\mu(x)\, ; \essinf_{x\in\Om} (\la(x)+ 2 \mu(x))\Big\} .
\end{equation}

Let $v \in [\Hspt^{1}(\Om)]^2$. For elasticity operator the left hand side of \eqref{eq:cond1str}
becomes
\begin{equation}\label{eq:tesielas}
\begin{gathered}
  \int_{\Om}\Big(
   ( \mu-\kappa)  |\nabla v|^2 + \la (\dive v)^{2}  
    + \mu \sum_{k,j}\de_{k} v_{j} \,  \de_{j} v_{k} \\
    -\LA^{2}(|v|)\left[ (\mu-\kappa) |\nabla|v||^{2} + (\la+\mu)|v|^{-2} | v_h \de_{h}|v||^2\right] \Big)dx \, .
\end{gathered}
\end{equation}

 Following the ideas used in \cite{CM2006}, given $v \in [\Hspt^{1}(\Om)]^2$,  we define
\begin{gather*}
X_{1}=|v|^{-1}(v_{1}\de_{1}|v|+v_{2}\de_{2}|v|),\quad 
X_{2}=|v|^{-1}(v_{2}\de_{1}|v|-v_{1}\de_{2}|v|) \cr
Y_{1}=|v|[\de_{1}(|v|^{-1}v_{1})+\de_{2}(|v|^{-1}v_{2})],
\quad
Y_{2}=|v|[\de_{1}(|v|^{-1}v_{2})-\de_{2}(|v|^{-1}v_{1})]
\end{gather*}
on the set $\Om_{0}=\{x\in\Om\ |\ v(x)\neq 0\}$.  We have $|v|^{-2} | v_h \de_{h}|v||^2= X_{1}^2$ and,
as it was proved  in  \cite[p.245]{CM2006}, 
\begin{gather*}
 |\nabla v|^2=X_{1}^2+ X_{2}^2 +Y_{1}^2+ Y_{2}^2; \quad  (\dive v)^{2}=(X_{1}+Y_{1})^{2};\\
|\nabla|v||^{2} = X_{1}^2+ X_{2}^2. 
\end{gather*}

We have also
\begin{equation}\label{eq:also}
 \sum_{k,j}\de_{k} v_{j} \de_{j}v_{k} = (\dive v)^{2} + 2(\de_{1}v_{2}\de_{2} v_{1} - \de_{1}v_{1}\de_{2} v_{2}) =
 (X_{1}+Y_{1})^{2} - 2(X_{1}Y_{1} + X_{2}Y_{2}).
\end{equation}

By means of these equalities,  the integral \eqref{eq:tesielas} can be written as
\begin{equation} \label{eq:tesielas2}
\begin{gathered}
    \int_{\Om_{0}}\Big( ( \la+2\mu-\kappa)[1-\LA^{2}(|v|)]X_{1}^{2} + 
    2\la X_{1}Y_{1} + ( \la+2\mu-\kappa)Y_{1}^{2}\Big)dx
    \\
    +
   \int_{\Om_{0}}\Big( (\mu-\kappa)[1-\LA^{2}(|v|)]X_{2}^{2} - 2 \mu X_{2}Y_{2} + (\mu-\kappa)Y_{2}^{2}\Big)dx\, .
       \end{gathered}
\end{equation}

Define
$$
\ga(x)=\mu(x) \frac{\la(x)+\mu(x)}{\la(x)+3\mu(x)}
$$
and rewrite \eqref{eq:tesielas2} as
\begin{equation} \label{eq:tesielas3}
\begin{gathered}
 \int_{\Om_{0}}\Big( (\mu-\kappa)[1-\LA^{2}(|v|)]X_{2}^{2} - 2 \ga X_{2}Y_{2} + (\mu-\kappa)Y_{2}^{2}\Big)dx  \\
 + 2  \int_{\Om_{0}} (\ga-\mu) (X_{1}Y_{1} +X_{2}Y_{2})\, dx \\
   +  \int_{\Om_{0}} \!  \Big( ( \la+2\mu-\kappa)[1-\LA^{2}(|v|)]X_{1}^{2} + 
    2(\la+\mu-\ga) X_{1}Y_{1} + ( \la+2\mu-\kappa)Y_{1}^{2}\Big)dx.
       \end{gathered}
\end{equation}

We claim that
\begin{equation}\label{eq:disgamma}
\ga^{2} =\mu^{2}\left( \frac{\la+\mu}{\la+3\mu}\right)^{2} < (\mu-\kappa)^{2} (1-\LA^{2}_{\infty}) \quad \text{a.e.}\, .
\end{equation}

Indeed \eqref{eq:defkappa} leads to $2\kappa(1-\LA^{2}_{\infty})<\delta \mu$ a.e., which implies 
$$
(2\mu \kappa - \kappa^{2}) (1-\LA^{2}_{\infty})  < \delta \mu^{2} \quad \text{ a.e.}\, .
$$
Since in view of \eqref{eq:delta}
$$
\mu^{2}\left(\frac{\la+\mu}{\la+3\mu}\right)^{2} < \mu^{2}(1-\LA^{2}_{\infty}-\delta) \quad  \text{ a.e.}\, ,
$$
 inequality \eqref{eq:disgamma} follows.
By similar arguments one can prove that
\begin{equation}\label{eq:disgamma2}
(\la+\mu-\ga)^{2} =(\la+2\mu)^{2}\left( \frac{\la+\mu}{\la+3\mu}\right)^{2} < 
(\la+2\mu-\kappa)^{2} (1-\LA^{2}_{\infty}) \quad \text{a.e.}\, .
\end{equation}

Inequalities \eqref{eq:disgamma} and \eqref{eq:disgamma2} show that
$$
(\mu-\kappa)[1-\LA^{2}(|v|)]X_{2}^{2} - 2 \ga X_{2}Y_{2} + (\mu-\kappa)Y_{2}^{2} \geq 0   \quad \text{a.e.}
$$
for any $X_{2}$, $Y_{2}$ and 
$$
( \la+2\mu-\kappa)[1-\LA^{2}(|v|)]X_{1}^{2} + 
    2(\la+\mu-\ga) X_{1}Y_{1} + ( \la+2\mu-\kappa)Y_{1}^{2} \geq 0  \quad \text{a.e.}
$$
for any $X_{1}$, $Y_{1}$. Therefore, keeping in mind \eqref{eq:tesielas3}, we can write
\begin{equation}\label{eq:tesielas4}
\begin{gathered}
  \int_{\Om}\Big(
   ( \mu-\kappa)  |\nabla v|^2 + \la (\dive v)^{2}  
    + \mu \sum_{k,j}\de_{k} v_{j} \,  \de_{j} v_{k}  \\
    -\LA^{2}(|v|)\left[ (\mu-\kappa) |\nabla|v||^{2} + (\la+\mu)|v|^{-2} | v_h \de_{h}|v||^2\right] \Big)dx
       \\
    \geq 2  \int_{\Om_{0}} (\ga-\mu) (X_{1}Y_{1} +X_{2}Y_{2})\, dx\, .
\end{gathered}
\end{equation}

Since (see \eqref{eq:also})
$$
 2(X_{1}Y_{1} + X_{2}Y_{2}) = (\dive v)^{2} - \sum_{k,j}\de_{k} v_{j} \de_{j}v_{k} 
 $$
 and
 $$
 \ga-\mu = -2 \frac{\mu^{2}}{\la + 3\mu}\, ,
 $$
 the last integral in \eqref{eq:tesielas4} can be written as
 $$
 2 \int_{\Om}\frac{\mu^{2}}{\la + 3\mu}\, \Big( \sum_{k,j}\de_{k} v_{j} \de_{j}v_{k} - (\dive v)^{2}\Big)dx\, .
 $$
 
 If $v\in [\Cspt^{\infty}(\Om)]^{2}$  and we consider $\mu^{2}/(\la + 3\mu)$ as a distribution $f$, we have
\begin{gather*}
\int_{\Om}f  \de_{k} v_{j} \de_{j}v_{k} \, dx = - \int_{\Om} \de_{k}f\, v_{j} \de_{j}v_{k}\, dx  - \int_{\Om}f\, v_{j} \de_{kj}v_{k} \, dx =\\
   - \int_{\Om} \de_{k}f\, v_{j} \de_{j}v_{k} \, dx+ \int_{\Om}\de_{j}f\, v_{j} \de_{k}v_{k} \, dx+ \int_{\Om}f\, \de_{j}v_{j} \de_{k}v_{k}\, dx
\end{gather*}
and then
$$
\int_{\Om}f   \Big(\sum_{k,j} \de_{k} v_{j} \de_{j}v_{k}-  (\dive v)^{2}\Big)dx = 
\sum_{k,j} \int_{\Om} \de_{k}f \, ( v_{k} \de_{j}v_{j} - v_{j} \de_{j}v_{k} )dx\, .
$$

Thanks to a result by Maz'ya and Verbitsky \cite[Lemma 4.9, p.1315]{MV2006} (see also \cite{MV2020}) we have the
commutator inequality
$$
\left| \sum_{k,j} \int_{\Om} \de_{k}f \, ( v_{k} \de_{j}v_{j} - v_{j} \de_{j}v_{k} )dx \right| \leq 
C_{0}\, \Vert f\Vert_{BMO}\, \Vert \nabla v_{1}\Vert \, \Vert \nabla v_{2}\Vert\, .
$$

Therefore
\begin{gather*}
\left| 2 \int_{\Om}\frac{\mu^{2}}{\la + 3\mu}\, \Big( \sum_{k,j}\de_{k} v_{j} \de_{j}v_{k} - (\dive v)^{2}\Big)dx \right|   \\  \leq
C_{0}\, \Vert \mu^{2} (\la + 3\mu)^{-1} \Vert_{BMO}\, \Vert \nabla v\Vert^{2}
\end{gather*}
for any $v\in [\Cspt^{\infty}(\Om)]^{2}$. By density, the same inequality holds for any $v \in [\Hspt^{1}(\Om)]^2$.

From \eqref{eq:tesielas4} it follows
\begin{equation}\label{eq:ultimadis}
\begin{gathered}
  \int_{\Om}\Big(
   \Big( \mu-\frac{\kappa}{2}\Big)  |\nabla v|^2 + \la (\dive v)^{2}  
    + \mu \sum_{k,j}\de_{k} v_{j} \,  \de_{j} v_{k}  \\  
    -\LA^{2}(|v|)\left[  \Big( \mu-\frac{\kappa}{2}\Big) |\nabla|v||^{2} + (\la+\mu)|v|^{-2} | v_h \de_{h}|v||^2\right] \Big)dx
      \\  \geq
     \frac{\kappa}{2} \int_{\Om}\Big( |\nabla v|^2 - \LA^{2}(|v|) \,  |\nabla|v||^{2}\Big)dx \\
       - C_{0}\,  \Vert \mu^{2} (\la + 3\mu)^{-1} \Vert_{BMO} \int_{\Om}  |\nabla v|^2 dx\, .
\end{gathered}
\end{equation}

If
\begin{equation}\label{eq:BMOineq}
 \Vert \mu^{2} (\la + 3\mu)^{-1} \Vert_{BMO} \leq \frac{\kappa}{2\, C_{0}}\, (1-\LA_{\infty}^{2})
\end{equation}
we have (see also \eqref{eq:nablav2})
\begin{gather*}
C_{0}\,  \Vert \mu^{2} (\la + 3\mu)^{-1} \Vert_{BMO} \int_{\Om}  |\nabla v|^2 dx \leq
\frac{\kappa}{2}\, (1-\LA_{\infty}^{2})\int_{\Om}  |\nabla v|^2 dx  \\
\leq \frac{\kappa}{2} \int_{\Om}\Big(|\nabla v|^2 - \LA^{2}(|v|) \,  |\nabla|v||^{2}\Big)dx
\end{gather*}
and the right hand side of \eqref{eq:ultimadis} is nonnegative. This means
that the operator $E-(\kappa/2)\Delta$ is $L^{\Phi}$-dissipative, which proves the theorem.
\end{proof}

Combining theorems \ref{th:4} and \ref{th:5}, we have immediately the following necessary and sufficient condition.
\begin{theorem}\label{th:6}
Assume that the $BMO$ seminorm of the function $ \mu^2\, (\la + 3\mu)^{-1}$ is sufficiently small.
The elasticity operator \eqref{opelast} is strict  $L^{\Phi}$-dissipative if and only if
the stric inequality \eqref{eq:condnuova} holds.
\end{theorem}

\begin{remark}
If $\la$ and $\mu$ are constant,  the $BMO$ seminorm of the function 
$ \mu^2\, (\la + 3\mu)^{-1}$ is zero and  then the strict inequality \eqref{eq:condnuova} is necessary and
sufficient for the strict  $L^{\Phi}$-dissipativity of elasticity  operator \eqref{opelast}.
\end{remark}

\begin{remark}\label{rem:2}
Condition \ref{item6} on the function $\vf$ is used only in the necessity part of
Theorem \ref{th:6}. Therefore, if  \ref{item6} it is not satisfied, 
the sufficiency part of Theorem \ref{th:6} holds, where $\LA^{2}_{\infty}=\sup_{t>0} \LA^{2}(t)$.
\end{remark}

\section{Some applications}\label{sec:applic}

In this section we show two applications of the theory we have developed.
In particular we obtain regularity results for energy solutions of 
Dirichlet problem for Lam\'e system. In these results the energy solution 
which \textit{a priori} belongs to the Sobolev space $H^1(\Om)$, 
actually satisfy  higher integrability conditions, provided that the 
right hand sides are better than $H^{-1}(\Om)$, the dual space of the
energy space (see, for instance, \cite[Sect. 1.1.15]{mazyasobolev}).
%vector of body forces
%satisfies a certain condition, 
%slightly more restrictive than belonging  to $L^2(\Om)$.

\subsection{The $N$-dimensional case ($N\geq 3$)}

We prove the following result wich concerns the $N$-dimensional Lam\'e system ($N\geq 3$)
with constant coefficients. As usually these constants are supposed to
 satisfy the inequalities: $\mu>0, \la+2\mu>0$.

First we give the following sufficient condition for the strict $L^\Phi$-dissipativity of
the Lam\'e  operator \eqref{opelastconst}.

\begin{theorem}\label{th:7}
     Let $\Om$ be a domain in $\R^N$ and suppose that
\begin{equation}\label{eq:obvious}
\LA_{\infty}^2<
\begin{cases}
\mu/(\la+2\mu), & \text{if  $\la+\mu>0$} ;   \\  
(\la+2\mu)/\mu, & \text{if  $\la+\mu \leq 0$} .
\end{cases}
\end{equation}
Then the Lam\'e operator \eqref{opelastconst} is strictly $L^\Phi$-dissipative.
\end{theorem}
\begin{proof}
Choose $\kappa$ such that $0<\kappa < \min\{\mu,\, \la+2\mu\}$ and
$$
\LA_{\infty}^2<
\begin{cases}
(\mu-\kappa)/(\la+2\mu-\kappa), & \text{if  $\la+\mu>0$} ;   \\  
(\la+2\mu-\kappa)/(\mu-\kappa), & \text{if  $\la+\mu \leq 0$} .
\end{cases}
$$

Setting 
$\la'=\la+\kappa$, $\mu'=\mu-\kappa$,
we can write
$$
\LA_{\infty}^2<
\begin{cases}
\mu'/(\la'+2\mu'), & \text{if  $\la'+\mu'>0$} ;   \\  
(\la'+2\mu')/\mu', & \text{if  $\la'+\mu' \leq 0$} .
\end{cases}
$$

Note that $\mu'>0$, $\la'+2\mu' > 0$. 
By repeating the arguments we have used in
in \cite[p.126]{CM2013} for the $L^p$-dissipativity,  we find that the operator
$$
E'u = \mu'\D u + (\la'+\mu')\nabla \dive u\, ,
$$
is $L^\Phi$-dissipative.
Since the  $L^\Phi$-dissipative operator $E'$ coincide with $E-\kappa\D$, 
 Corollary \ref{co:co1} shows that $E$ is strictly $L^\Phi$-dissipative.
\end{proof}

\begin{remark}
Obviously condition \eqref{eq:obvious} can be reformulated in terms of the Poisson ratio $\nu=
\la/(2(\la+\mu))$ as
$$
\LA_{\infty}^2 \leq \begin{cases}(1-2 \nu)/(2(1-\nu)) & \text { if } \nu<1 / 2 \\ 2(1-\nu)/(1-2 \nu) & \text { if } \nu>1.\end{cases}
$$
\end{remark}

\begin{theorem}\label{th:8}
    Let $\Om\subset \R^N$ be a bounded domain and let $p\geq 2$ such that 
    \begin{equation}\label{eq:pN}
\left(1 - \frac{2}{p}\right)^2 <
\begin{cases}
\mu/(\la+2\mu), & \text{if  $\la+\mu>0$} ;   \\  
(\la+2\mu)/\mu, & \text{if  $\la+\mu \leq 0$} .
\end{cases}
\end{equation}
Consider the Dirichlet problem for the Lam\'e operator  \eqref{opelastconst}
\begin{equation}\label{eq:Dir}
\begin{cases}
u\in \Hspt^1(\Om)\\
Eu=\Dive F  & \text{in $\ \Om$},  \\ 
u=0 & \text{on $\partial\Om$},
\end{cases}
\end{equation}
where $F=\{F_{ij}\}$ is a given matrix in $[L^{\frac{Np}{N+p-2}}(\Om)]^{N^2}$
and $\Dive F$ denotes the vector whose $j$-th component is $\de_{i}F_{ij}$. 
Then the solution $u$ satisfies the inequality
\begin{equation}\label{eq:exdis1}
\int_{\Om}|\nabla u|^2\, (|u|^{p-2} + 1)\, dx < +\infty\, .
\end{equation}
In particular,
\begin{equation}\label{eq:exdis2}
\int_{\Om}|\nabla u|^2\, |u|^{p-2} \, dx \leq C 
\left(\int_{\Om}|F|^{\frac{Np}{N+p-2}}dx\right)^{\frac{N+p-2}{N}} 
\end{equation}
where the constant $C$ does not depend on $u$.
Moreover the solution $u$ belongs to $[L^{\frac{Np}{N-2}}(\Om)]^N$.
\end{theorem}

\begin{proof}
% In view of Theorem \ref{th:7}, the strict inequality
% \eqref{eq:pN} is a sufficient condition for the strict $L^p$-dissipativity
% of the operator $E$, and then inequality \eqref{eq:defdiss0str} holds
% with $\vf(t)=t^{p-2}$:
% \begin{equation}\label{eq:lamepell}
% \begin{gathered}
% \int_{\Om}[\mu\, \lan \nabla v, \nabla(|v|^{p-2}v)\ran \,+\, (\la+\mu) (\dive v)(\dive(|v|^{p-2}v))]\, dx \\
%  \geq 
% \kappa \int_{\Om}| \nabla(\sqrt{|v|^{p-2}}\, v)|^2 dx
% \end{gathered}
% \end{equation}
% for any $v\in \Hspt^1(\Om)$.
Saying that  $u$ is solution of problem \eqref{eq:Dir}  means
\begin{equation}\label{eq:weak}
\int_{\Om} [\mu\, \lan \nabla u, \nabla v\ran \,+\, (\la+\mu) (\dive u)(\dive v)]\, dx =
\int_{\Om} F_{ij}\, \de_{i}v_{j}\, dx
\end{equation}
for any $v\in [\Hspt^1(\Om)]^N$.
The existence and the uniqueness of the solution $u\in \Hspt^1(\Om)$ is 
 guaranteed by classic results, because  $F\in [L^{\frac{Np}{N+p-2}}(\Om)]^{N^2}
\subset \left[L^{2}(\Om)\right]^{N^2}$,  the number
$Np/(N+p-2)$ being greater than or equal to $2$.

Let now $k>0$ and define
$$
\vf_{k}(t) = 
\begin{cases}
t^{p-2},  & \text{if $0\leq t <k-1$; }  \\ 
[\ro_k(t)]^{p-2},  & \text{if $k-1\leq t \leq k$; }  \\ 
(k-1/2)^{p-2},  & \text{if $t > k$, }
\end{cases}
$$
where $\ro_k(t)= -(t-k+1)^2/2 + t$. We note that
$$
\ro_k(k-1)=k-1,\ \ro_k(k)=k-1/2,\ \ro'_k(k-1)=1,\ \ro'_k(k)=0,
$$
and then $\vf_{k}\in C^1(0,+\infty)$. It is easy to check that $\vf_{k}$ satisfies
also conditions \ref{item2}- \ref{item5}. 
% Concerning \ref{item6}, we can write
% $$
% \frac{t \vf'_{k}(t)}{\vf_{k}(t)} =
% \begin{cases}
% p-2,  & \text{if $0\leq t <k-1$; }  \\ 
% \frac{(p-2)t\ro'_k(t)}{[\ro_k(t)]^{p-2}},  & \text{if $k-1\leq t \leq k$; }  \\ 
% 0,  & \text{if $t > k$, }
% \end{cases}
% $$
Let us consider now the functional $\Phi_k$-dissipativity related to $\vf_{k}$ and
$$
\LA_{k}\left(s\sqrt{\vf_{k}(s)}\right)= - \frac{s\, \vf_{k}'(s)}{s\,\vf_{k}'(s)+2\, 
\vf_{k}(s)}\, .
$$
Since 
$$
\frac{t\ro'_k(t)}{\ro_k(t)}= \frac{2t(k-t)}{-(t-k+1)^2+2t}
$$
 decreases from $1$ to $0$ in $[k-1,k]$, the function
 \begin{equation}\label{eq:vfk}
\frac{t \vf'_{k}(t)}{\vf_{k}(t)}=
\begin{cases}
p-2,  & \text{if $0\leq t <k-1$; }  \\ 
(p-2)t\ro'_k(t)/ \ro_k(t),  & \text{if $k-1\leq t \leq k$; }  \\ 
 0,  & \text{if $t > k$, }
 \end{cases}
\end{equation}
decreases from $(p-2)$ to $0$ in $(0,+\infty)$.
% we have 
% $$
% \LA_{k}\left(s\sqrt{\vf_{k}(s)}\right)=
% \begin{cases}
%  -(1-2/p),  & \text{if $0\leq t <k-1$; }  \\ 
% - (p-2)t\ro'_k(t)/((p-2)t\ro'_k(t) + 2 \ro_k(t)),  & \text{if $k-1\leq t \leq k$; }  \\ 
%  0,  & \text{if $t > k$, }
%  \end{cases}
% $$
% 
% Since the function
% $$
% \frac{t\ro'_k(t)}{\ro_k(t)}= \frac{2t(k-t)}{-(t-k+1)^2+2t}
% $$
%  decreases from $1$ to $0$ in $[k-1,k]$, 
%  the function $- (p-2)t\ro'_k(t)/((p-2)t\ro'_k(t) + 2 \ro_k(t))$ increases from $-(1-2/p)$ to $0$ in the
%  same interval. 
 This implies that 
 $\LA^{2}_{k}\left(s\sqrt{\vf_{k}(s)}\right)$ is decreasing in $(0,+\infty)$ and then
$\LA_{k}^2 \leq (1-2/p)^2$.  Inequality \eqref{eq:pN} and
Theorem \ref{th:7} show that $E$ is $L^{\Phi_k}$-dissipative with the same constant
$\kappa$, and then
\begin{gather*}
\int_{\Om}[\mu\, \lan \nabla v, \nabla(\vf_{k}(|v|)\,v)\ran \,+\, (\la+\mu) (\dive v)(\dive(\vf_{k}(|v|)\, v))]\, dx \\
 \geq 
\kappa \int_{\Om} |\nabla(\sqrt{\vf_{k}(|v|)}\, v)|^{2} dx
\end{gather*}
for any $v\in [\Hspt^1(\Om)]^N$. Since $\vf_{k}(|u|)\,u$ belongs to $[\Hspt^1(\Om)]^N$, 
this inequality and \eqref{eq:weak} lead to
\begin{equation}\label{eq:disug1}
  \kappa \int_{\Om}| \nabla(\sqrt{\vf_{k}(|u|)}\, u)|^2 dx \leq
\int_{\Om} |F_{ij}|\, |\de_{i}(\vf_{k}(|u|)\, u_{j})|\, dx.
\end{equation}

We can write
\begin{gather*}
\de_{i}(\vf_{k}(|u|)\, u_{j}) = \de_{i}(\sqrt{\vf_{k}(|u|)}\,\sqrt{\vf_{k}(|u|)}\, u_{j}) \\
=
\sqrt{\vf_{k}(|u|)}\, u_{j} \, \de_{i}(\sqrt{\vf_{k}(|u|)}\,) + 
 \sqrt{\vf_{k}(|u|)}\,\de_{i}(\sqrt{\vf_{k}(|u|)}\, u_{j})\, .
\end{gather*}

By Cauchy inequality we get 
\begin{gather*}
\int_{\Om} |F_{ij}|\, |\de_{i}(\vf_{k}(|u|)\, u_{j})|\, dx \\
%\int_{\Om} |F_{ij}|\,  \big[\sqrt{\vf_{k}(|u|)}\, u_{j} \, \de_{i}(\sqrt{\vf_{k}(|u|)}\, ) + 
% \sqrt{\vf_{k}(|u|)}\,\de_{i}(\sqrt{\vf_{k}(|u|)}\, u_{j})\big]\, dx\\
\leq
\left(\int_{\Om} |F|^{2} \vf_{k}(|u|)\, dx \right)^{1/2}\bigg[
 \left(\int_{\Om} |u|^{2}\, |\nabla (\sqrt{\vf_{k}(|u|)}|^{2} dx \right)^{\frac{1}{2}}\\
 + 
 \left(\int_{\Om} |\nabla(\sqrt{\vf_{k}(|u|)}\, u)|^{2} dx \right)^{\frac{1}{2}}\bigg] ,
\end{gather*}
where $|F|=(\sum_{i,j}^{1,N}|F_{ij}|^2)^{1/2}$.
In view of the decrease of \eqref{eq:vfk} we find
$$
|u|^{2}\, |\nabla (\sqrt{\vf_{k}(|u|)}\, |^{2}= |u|^{2}\, (\vf'_{k}(|u|))^{2}/(4\,\vf_{k}(|u|))  |\nabla u|^2
\leq (1-p/2)^2\, \vf_{k}(|u|)  |\nabla u|^2,
$$
from which it follows
$$
|u|^{2}\, |\nabla (\sqrt{\vf_{k}(|u|)}\, |^{2}
%\leq (1-p/2)^2 \vf_{k}(|u|) |\nabla u|^2
\leq (p-2)^2\, |\nabla (\sqrt{\vf_{k}(|u|)}\, u)|^{2}
$$
(see \eqref{eq:lastineq}).  Inequality \eqref{eq:disug1} leads to
\begin{equation}\label{eq:vk}
  \int_{\Om}| \nabla v_{k}|^2 dx \leq  C
 \int_{\Om} |F|^{2} \vf_{k}(|u|)\, dx \, ,
\end{equation}
where $v_{k}= \sqrt{\vf_{k}(|u|)}\, u$. Here and in the sequel the same symbol $C$ denotes different constants which do not depend
on $v_k$.

Setting $\al=Np/((N-2)(p-2))$, by H\"older inequality we have
$$
 \int_{\Om} |F|^{2} \vf_{k}(|u|)\, dx \leq
 \left(\int_{\Om} \vf_k(|u|)^\al dx\right)^{\frac{1}{\al}}
  \left(\int_{\Om} |F|^{2\al'} dx\right)^{\frac{1}{\al'}}
$$
where $\al'=Np/(2(N+p)-4)$.

We claim
\begin{equation}\label{eq:ineqvk}
\vf_{k}(|u|)\leq |v_{k}|^{2(p-2)/p}.
\end{equation}
In fact, we have
$$
 |v_{k}|^{\frac{2(p-2)}{p}}=
 \begin{cases}
|u|^{p-2},  & \text{if $0\leq |u| <k-1$; }  \\ 
(\ro_k(|u|)^{\frac{(p-2)^2}{p}}\, |u|^{\frac{2(p-2)}{p}},  & \text{if $k-1\leq |u| \leq k$; }  \\ 
 (k-1/2)^{\frac{(p-2)^2}{p}}\,  |u|^{\frac{2(p-2)}{p}}  ,  & \text{if $|u| > k$, }
 \end{cases}
$$
and inequality \eqref{eq:ineqvk} for $|u|<k-1$ or $|u| > k$ easily follows from
the identity $(p-2)^2 + 2(p-2)= p (p-2)$. If $k-1\leq |u| \leq k$ we note that
$\ro_{k}(t)\leq t$ in $[k-1,k]$ and then
\begin{gather*}
\vf_{k}(|u|)=[\ro_{k}(|u|)]^{p-2} = [\ro_{k}(|u|)]^{\frac{(p-2)^2}{p}+\frac{2(p-2)}{p} } \\
\leq [\ro_{k}(|u|)]^{\frac{(p-2)^2}{p}} |u|^{\frac{2(p-2)}{p} }=|v_{k}|^{\frac{2(p-2)}{p}}.
\end{gather*}

Therefore
$$
 \int_{\Om} |F|^{2} \vf_{k}(|u|)\, dx \leq
 \left(\int_{\Om} |v_{k}|^{\frac{2N}{N-2}} dx\right)^{\frac{1}{\al}}
  \left(\int_{\Om} |F|^{2\al'} dx\right)^{\frac{1}{\al'}}
$$
and in view of Sobolev imbedding theorem (see, for instance, \cite[Sect. 2.3.5]{mazyasobolev},
where also the the best constant is  determined)
we obtain
$$
\int_{\Om} |F|^{2} \vf_{k}(|u|)\, dx \leq C  \left(\int_{\Om} |\nabla v_{k}|^{2} dx\right)^{\frac{p-2}{p}}
  \left(\int_{\Om} |F|^{2\al'} dx\right)^{\frac{1}{\al'}}.
$$

Inequality \eqref{eq:vk} implies
$$
\int_{\Om} |\nabla v_{k}|^{2} dx \leq C
 \left(\int_{\Om} |F|^{\frac{Np}{N+p-2}} dx\right)^{\frac{N+p-2}{N}}
$$

From \eqref{eq:lastineq} we have also
$$
4\, |\nabla v_{k}|^{2} \geq \vf_k(|u|)\, |\nabla u|^2
$$
and then
$$
\int_{\Om} |\nabla u|^2\,  \vf_k(|u|)\,  dx \leq C
 \left(\int_{\Om} |F|^{\frac{Np}{N+p-2}} dx\right)^{\frac{N+p-2}{N}}.
$$

By Fatou's Lemma, letting $k\to +\infty$, we obtain \eqref{eq:exdis2} and  \eqref{eq:exdis1} follows immediately.
Recalling Remark \ref{rem:lemmino}, we have also
$$
\int_{\Om}|\nabla(|u|^{\frac{p-2}{2}}\, u)|^{2} \,  dx \leq C
 \left(\int_{\Om} |F|^{\frac{Np}{N+p-2}} dx\right)^{\frac{N+p-2}{N}}
$$
 and  then
$u$ belongs to $[L^{\frac{Np}{N-2}}(\Om)]^N$, because of the Sobolev imbedding theorems.
\end{proof}

Let us show that the previous result can  be extended to a certain class of Lam\'e operators with variable Lam\'e coefficients.

\begin{corollary}
    Let $S$ be the Lam\'e operator with variable coefficients of the  form \eqref{eq:A}, where
$$
a^{hk}_{ij}(x)= (\la+\ep(x)) \del_{ih}  \del_{jk} +( \mu+ \si(x))  ( \del_{ij}  \del_{hk}   + \del_{ik}  \del_{hj})\, .
$$
Here $\la$ and $\mu$ are constant, while $\ep(x)$ and $\si(x)$ are $L^{\infty}$ functions.
Suppose $p\geq 2$ is such that condition \eqref{eq:pN} holds and 
$\Vert \,  |\ep|+|\si| \, \Vert_{L^\infty(\Om)}$ is sufficiently small. Then
the solution $u$ of the Dirichlet problem \eqref{eq:Dir}, where $F$ belongs to $[L^{\frac{Np}{N+p-2}}(\Om)]^{N^2}$, satisfies the inequalities \eqref{eq:exdis1}, \eqref{eq:exdis2}. Moreover, the solution $u$ belongs to $[L^{\frac{Np}{N-2}}(\Om)]^N$.
\end{corollary}
\begin{proof}
First, let us prove that the operator $S$ is strict $L^p$-dissipative. By Lemma \ref{le:7}
the strict $L^p$-dissipativity of $S$ can be written in terms of the vector function
$v=|u|^\frac{p-2}{2}u$ as follows
\begin{equation}\label{eq:dislamevarN}
\begin{gathered}
\int_{\Om}\Big(
   (\mu+\si(x))  |\nabla v|^2 + (\la+\ep(x)) (\dive v)^{2}  
    + (\mu+\si(x)) \sum_{k,j}\de_{k} v_{j} \,  \de_{j} v_{k} \\
    -(1-2/p)^{2}\left[ (\mu+\si(x)) |\nabla|v||^{2} + (\la+\ep(x)+ \mu+\si(x))|v|^{-2} | v_h \de_{h}|v||^2\right] \Big)dx \\
    \geq \kappa \int_{\Om}|\nabla v|^2 dx, \quad \forall\ v\in [\Hspt^1(\Om)]^N,
\end{gathered}
\end{equation}
where $\kappa$ is a positive constant. On the other hand, there exists a constant $C$, which does not depend on
$\ep$, $\si$ and $v$, 
such that
\begin{gather*}
\Big|\int_{\Om}\Big(
   \si(x)  |\nabla v|^2 + \ep(x) (\dive v)^{2}  
    + \si(x) \sum_{k,j}\de_{k} v_{j} \,  \de_{j} v_{k} \\
    -(1-2/p)^{2}\left[ \si(x) |\nabla|v||^{2} + (\ep(x) +\si(x))|v|^{-2} | v_h \de_{h}|v||^2\right] \Big)dx \Big|\\
    \leq C\, \Vert\,  |\ep|+|\si| \, \Vert_{L^\infty(\Om)}
  \int_{\Om}|\nabla v|^2 dx,  \quad \forall\ v\in [\Hspt^1(\Om)]^N.
 \end{gather*}
 Since  $\la$ and $\mu$ satisfy condition \eqref{eq:pN}, we have the strict $L^p$-dissipativity 
 of the Lam\'e operator \eqref{opelastconst} and then
 there exists $\kappa_{0}>0$ such that
 \begin{gather*}
\int_{\Om}\Big(
   \mu  |\nabla v|^2 + \la (\dive v)^{2}  
    + \mu \sum_{k,j}\de_{k} v_{j} \,  \de_{j} v_{k} \\
    -(1-2/p)^{2}\left[ \mu |\nabla|v||^{2} + (\la+\mu)|v|^{-2} | v_h \de_{h}|v||^2\right] \Big)dx \\
    \geq \kappa_{0} \int_{\Om}|\nabla v|^2 dx,  \quad \forall\ v\in [\Hspt^1(\Om)]^N.
\end{gather*}
 If $\ep$ and $\si$ are such that
$$
\Vert \,  |\ep|+|\si| \, \Vert_{L^\infty(\Om)} \leq \frac{\kappa_{0}}{2C}\, ,
$$
inequality \eqref{eq:dislamevarN} holds with
$\kappa=\kappa_{0}/2$.  Since the operator $S$ is strict
$L^p$-dissipative, then the proof of Theorem \ref{th:8} can be repeated and the result follows.
\end{proof}

\subsection{The $2$-dimensional case}

Before giving our result for $N=2$, we recall a couple of facts concerning Orlicz spaces.
For the general theory of these spaces we refer to the monographs \cite{krasno} and \cite{raoren}.

Let $(M,N)$ be a  complementary pair of Young's functions (see, e.g., \cite[p.6]{raoren})
defined for $t\geq 0$. 
The Orlicz space $\elle_M(\Om)$ is defined as the class of measurable functions defined in $\Om$ such that there exists $\al>0$ satisfying
$$
\int_{\Om} M(\al\, |u|)\, dx <+\infty.
$$
In the space $\elle_M(\Om)$ we can introduce two norms:
$$
\Vert u\Vert_{\elle_M(\Om)} = \sup\left\{ \int_{\Om} u\, v\, dx\ \Big| \ \int_{\Om} N(|v|)\, dx \leq 1\right\}
$$
which is called Orlicz norm, and the Luxemburg  norm
$$
\vertiii{u}_{\elle_M(\Om)} = \inf \left\{\ \la>0\ \Big|\ \int_{\Om} M(|u|/\la)\, dx \leq 1 \right\}.
$$
The two norms are equivalent, because of the inequalities
\begin{equation}\label{eq:eqnorms}
\vertiii{u}_{\elle_M(\Om)}  \leq \Vert u\Vert_{\elle_M(\Om)}  \leq 2\,\vertiii{u}_{\elle_M(\Om)}
\end{equation}
for any $u\in \elle_M(\Om)$ (see, e.g.,  \cite[p.61]{raoren}).  We recall also that
 H\"older inequality holds in the following form (see, e.g.,  \cite[p.58]{raoren})
\begin{equation}\label{eq:holder}
 \int_{\Om}| u\, v|\, dx \leq  2\, \vertiii{u}_{\elle_M(\Om)}\, \vertiii{v}_{\elle_N(\Om)}\, .
\end{equation}

\begin{theorem}
    Let $\Om$  be a bounded domain in $\R^2$. Let $E$ be the Lam\'e operator
    \eqref{opelast} with coefficients in $L^\infty(\Om)$. Suppose 
$$
\left(1- \frac{2}{p}\right)^2 <  1 - \esssup_{x\in\Om}\left(\frac{\la + \mu}{\la+3\mu}\right)^2
$$
with $p\geq 2$ and that the BMO norm of the function $\mu^2/(\la+3\mu)$ satisfies inequality \eqref{eq:BMOineq}.
 Consider the Dirichlet problem \eqref{eq:Dir}, where the matrix $F$ is such that
 \begin{equation}\label{eq:HypF}
\int_{\Om} |F|^2 (\log(|F|+e))^{\frac{p-2}{p}} dx < +\infty
\end{equation}
Then the solution $u$ satisfies the inequality \eqref{eq:exdis1}.
In particular,
\begin{equation}\label{eq:exdis2n=2}
\int_{\Om}|\nabla u|^2\, |u|^{p-2} \, dx \leq K 
\vertiii{\, |F|^2\, }_{\elle_{\widetilde{N}}(\Om)}^{\frac{p}{2}}
\end{equation}
where 
$$
\widetilde{N}(t)=t\, (\log(t+e))^{\frac{p-2}{p}}
$$ 
and the constant $K$ does not depend on $u$.
  \end{theorem}
\begin{proof}
If $p=2$ the result is well known,  since in this case \eqref{eq:HypF} means $F\in [L^2(\Om)]^2$.
Suppose then $p>2$.  The assumptions on $p$ and Lam\'e coefficients assure that
$E$ is stric $L^p$-dissipative, because of Theorem \ref{th:5}.  As in Theorem \ref{th:8} we find  \eqref{eq:vk}, keeping in mind also Remark
\ref{rem:2}, .

The function $v_k$ being in $\Hspt^1(\Om)$, Yudovich-Trudinger-Moser inequality \cite{yudo,moser, trudinger} holds:
$$
\int_{\Om}e^{4\pi |v_k|^2}dx < +\infty \, .
$$
Setting $w_{k}=\vf_{k}(|u|)$ we have also
$$
\int_{\Om}e^{4\pi w_{k}^{p/(p-2)}}dx < +\infty \, ,
$$
because $w_{k}\leq |v_{k}|^{2(p-2)/p}$ (see \eqref{eq:ineqvk}). Set $M(t)=e^{4\pi t^{p/(p-2)}} -1$. Let $N(t)$ be  its complementary Young's function. We have
\begin{equation}\label{eq:Nt}
N(t)=t\, \left(\frac{\log(t+e)}{4\pi}\right)^{\frac{p-2}{p}}(1 + o(1))
\end{equation}
(as $t\to + \infty$). 
By the H\"older inequality \eqref{eq:holder}
\begin{equation}\label{eq:holderN}
\int_{\Om} |F|^{2}w_{k}\, dx \leq  2\, \vertiii{\, w_k\, }_{\elle_M(\Om)}\, \vertiii{\, |F|^{2}\, }_{\elle_N(\Om)}\, .
\end{equation}
Let us introduce now another Orlicz space $\elle_{M_{0}}(\Om)$, where
$M_{0}(t)=e^{4\pi t} -1$. We prove that
\begin{equation}\label{eq:ineqorlicz}
\vertiii{\, w_k\, }_{\elle_M(\Om)} \leq \vertiii{\, |v_k|^{2}\, }_{\elle_{M_{0}}(\Om)}^{\frac{p-2}{p}}\, .
\end{equation}
Take $\mu>0$ such that
\begin{equation}\label{eq:infproof}
\int_{\Om}(e^{4\pi \frac{|v_k|^2}{\mu}}-1)\, dx \leq 1.
\end{equation}
We have also
$$
\int_{\Om}(e^{4\pi \left(\frac{w_{k}}{\la}\right)^{\frac{p}{p-2}}}-1)\, dx \leq 1
$$
where $\la=\mu^{(p-2)/p}$. By definition of Luxemburg norm
$$
\vertiii{\, w_k\, }_{\elle_M(\Om)}\leq \la\, ,
$$
i.e.
$$
\vertiii{\, w_k\, }_{\elle_M(\Om)}^{\frac{p}{p-2}} \leq \mu.
$$
This being true for any $\mu$ satisfying \eqref{eq:infproof},
 we obtain inequality \eqref{eq:ineqorlicz} taking the infimum on the right hand side.

We claim now that
\begin{equation}\label{eq:mazyaimb}
\Vert\, |v|^{2}\, \Vert_{\elle_{M_{0}}(\Om)} \leq C \int_{\Om}|\nabla v|^2 dx \, .
\end{equation}

This inequality is a particular case of a general result proved by Maz'ya (see \cite[p.158]{mazyasobolev}). In view of this theorem, we can say that \eqref{eq:mazyaimb} is true
for any $v\in \Cspt^{\infty}(\Om)$ if and only if
 there exists a constant $\beta$ such that
\begin{equation}\label{eq:beta}
m(F)\, N_{0}^{-1}(1/m(F)) \leq \beta\, \text{cap}(F,\Om)
\end{equation}
for any compact set $F\subset \Om$. Here $m(F)$ denotes the Lebesgue measure of $F$  and  $\text{cap}(F,\Om)$ is the capacity 
of $F$ relative to $\Om$, i.e.
$$
\text{cap}(F,\Om) = \inf \left\{ \int_{\Om} |\nabla u|^2 dx\ \Big|\ u\in \Cspt^{\infty}(\Om),\, u\geq 1 \text{ on } F \right\}.
$$

We can write
$$
N_{0}^{-1}(t) = 4\pi t /\log(m(\Om)t+e)(1+o(1))
$$
and then
\begin{gather*}
m(F)\, N_{0}^{-1}(1/m(F)) = 4\pi /\log([m(\Om)/m(F)]+e)(1+o(1)) \\
\leq 4\pi /\log([m(\Om)/m(F)])(1+o(1)).
\end{gather*}

On the other hand, the following isocapacitary inequality holds (see \cite[p.148]{mazyasobolev})
$$
\text{cap}(F,\Om) \geq 4\pi  /\log([m(\Om)/m(F)])
$$
and \eqref{eq:beta} is proved. Thanks to Maz'ya's result, estimate \eqref{eq:mazyaimb} is valid
for any $v\in \Cspt^{\infty}(\Om)$ and then, by density, for any $v\in \Hspt^{1}(\Om)$.

From \eqref{eq:vk}, \eqref{eq:holderN}, \eqref{eq:ineqorlicz}, \eqref{eq:eqnorms}  and 
\eqref{eq:mazyaimb} it follows
$$
\int_{\Om} |\nabla u|^2\,  \vf_k(|u|)\,  dx \leq C
\vertiii{\, |F|^2\, }_{\elle_{{N}}(\Om)}^{\frac{p}{2}}
$$
Letting $k\to +\infty$ we obtain 
$$
\int_{\Om} |\nabla u|^2\,  |u|^{p-2}  dx \leq C
\vertiii{\, |F|^2\, }_{\elle_{{N}}(\Om)}^{\frac{p}{2}}
$$
and \eqref{eq:exdis2n=2} follows immediately from \eqref{eq:Nt}.
\end{proof}

 \section*{Funding}
 The second author has been supported by the RUDN University Strategic
Academic Leadership Program.

\end{document}